\documentclass[11pt]{amsart}

\usepackage{amscd,amsthm,amsfonts,amssymb, amsmath}
\usepackage{tikz-cd}
\usepackage[colorlinks=true]{hyperref}
\usepackage{fullpage}
\usepackage[all]{xy}
\usepackage{color}
\usepackage{comment}
\usepackage{enumitem}
\usepackage[mathscr]{euscript}
\usepackage{mathtools}
\usepackage{mathrsfs}

\numberwithin{equation}{section}

\newtheorem{theorem}{Theorem}[section]

\newtheorem{lemma}[theorem]{Lemma}

\newtheorem{corollary}[theorem]{Corollary}

\newtheorem{definition}[theorem]{Definition}

\newcommand{\GL}{\mathrm{GL}}

\newcommand{\sumh}{\mathop{\sum\nolimits^{h}}}

\newcommand{\Q}{\mathbb{Q}}

\newcommand{\R}{\mathbb{R}}

\subjclass[2020]{Primary: 11F72; Secondary: 11F66}
\keywords{Symmetric-square $L$-function, Voronoi summation, Petersson formula,
Kuznetsov formula, trace formula, beyond endoscopy, Maass forms.}

\title{A Trace Formula Derivation of the Functional Equation for Symmetric Square L-functions}

\author{Taiwang  Deng}
\address[T.D.]{Beijing Institute of Mathematical Sciences and Applications (BIMSA), Huairou District, 100084, Beijing\\
China}
\email{dengtaiw@bimsa.cn}

\author{Dongsheng Wu}
\address[D.W.]{Beijing Institute of Mathematical Sciences and Applications (BIMSA), Huairou District, 100084, Beijing\\
China}
\email{wds2021fall@bimsa.cn}
\date{}

\begin{document}

\begin{abstract}
This paper revisits the functional equation of the symmetric-square
$L$-function from the perspective of trace formulas. We show that, for
level-one holomorphic cusp forms and even Hecke--Maass cusp forms, this symmetry
can be recovered directly from the Petersson and Kuznetsov trace formulas. Although
the functional equation itself is classical, our aim is to provide a concrete
and transparent example of how the analytic structure of an $L$-function can
emerge from a comparison of the spectral and geometric sides of a trace
formula. The argument leads to an averaged reciprocity identity and treats the
holomorphic and Maass settings within the same general framework. In this way,
the paper offers a simple model for extending trace-formula methods beyond
standard $L$-functions.
\end{abstract}

\maketitle

\section{Introduction}
\subsection{Background and motivation}
Automorphic $L$ functions plays a central role in the classical Langlands program over global fields. It is known
that the meromorphic continuation and functional equations of
general automorphic $L$-function would imply Langlands' transfer conjecture when combined with suitable converse theorem, see \cite[Introduction]{N20} for more detailed explanations.

Automorphic $L$-functions admit several complementary constructions. A basic
prototype is the work of Godement and Jacquet \cite{GJ72}, where the standard
$L$-function on $\GL_n$ is realized by zeta integrals from which the
functional equation follows. For the symmetric square of $\GL_2$, the analytic theory
is classically provided by the lift of Gelbart and Jacquet \cite{GelJac78} base on previous work by Shimura, with
further developments on automorphy and cuspidality in the work of Kim and Shahidi \cite{KS02}. From the trace-formula side, Jacquet and Zagier \cite{JZ87} gave a
different model for extracting $L$-functions from spectral and geometric
expansions. There is also a proposal by Braverman and Kazhdan, aiming to prove the meromorphic continuation and functional equation of general automorphic $L$-function along the lines of Godement-Jacquet (see \cite{N20}).

Langlands' beyond endoscopy proposal \cite{Lan04} asks whether one can isolate
functorial transfers directly from trace formulas by inserting suitable
$L$-functions on the spectral side and then analyzing the resulting geometric
expressions. Since its introduction, this proposal has inspired a variety of approaches to beyond endoscopy and trace-formula methods. In the case of $\GL_2$, this point of view was developed by
Altu{\u{g}} in a series of works beginning with \cite{Alt15} and culminating, for
the standard representation, in \cite{Alt20}. A central feature there is the
combination of approximate functional equations, Poisson summation, and the
smoothing of orbital-integral singularities. 
Most recently the work \cite{DE26} has pushed the work of Altu{\u{g}} to the case of $\GL_3(\Q)$. 
Related perspectives include
Sakellaridis' work on beyond endoscopy for relative trace formulas \cite{Sak19},
the extension of Altu{\u{g}}'s smoothing method to $\GL_n$ by
\cite{GKMVW18}. We point out that in the work of Venkatesh \cite{Venk02}, the Kuznetsov Trace Formula is used, instead of the Arthur-Selberg Trace formula, to carry out the ideas put forward by Langlands. 

The aim of this paper is modest and explicit. We study the symmetric-square
$L$-function attached to level-one holomorphic cusp forms and even Hecke--Maass
cusp forms on $\GL_2$, and we recover its functional equation from trace-formula
identities of Petersson/Kuznetsov type. The key point is an explicit
archimedean transform $\mathcal T_\mu$ whose Mellin transform produces the
gamma factor of $\operatorname{Sym}^2$, which is introduced for Voronoi summation formulas for $\GL_3$ in \cite{MS06}. After Poisson summation in a quadratic
variable and a Weber--Schafheitlin calculation, the geometric side reorganizes
into a Voronoi-type reciprocity formula, from which the functional equation is
deduced. For the standard 
$L$-function, similar procedure
was first developed in \cite{H12}, which is carefully revisited in \cite{CL24}.
Our work shows that the beyond endoscopy
approach to $L$-functions can go beyond the standard case. 

None of the analytic properties proved here are new; rather, the point is to
record a concrete trace-formula derivation in the simplest symmetric-square
setting. To the best of our knowledge, this is the first trace-formula derivation of the functional equation of $\operatorname{Sym}^2$ $L$-functions via beyond endoscopy. Another new point compared to \cite{CL24} is that we apply the Mellin transform to the Voronoi summation formulas by introducing a new variable in the test functions. This derives the functional equations and entireness of $\operatorname{Sym}^2$ $L$-functions directly.

It would be desirable to extend the current proof to cover the general number field cases and to remove the restriction on levels.
This could be done via an adelic version of the current results, following the lines of \cite{Li24}.

\subsection{Main results}

\subsubsection{Holomorphic cusp forms}

Let $k\geqslant 2$ be an integer. Let $S_{2k}$ be the space of holomorphic cusp forms of weight $2k$ and level $1$, and let $B_{2k}$ be an orthogonal basis of Hecke eigenforms with Fourier expansions
\[
f(z)=\sum_{n\geqslant 1} a_f(n) n^{\frac{2k-1}{2}} e^{2\pi i nz},\qquad a_f(1)=1.
\]
For such an $f$ define the symmetric-square coefficients
\[
A_f(1,n)=A_f(n,1)=\sum_{d^2\mid n} a_f\bigl((n/d^2)^2\bigr).
\]
For a Schwartz function $g\in\mathcal{S}(0,\infty)$ and an integer $\ell\geqslant 1$ set
\[
I_\ell(g)=\sum_{f\in B_{2k}} \frac{\Gamma(2k-1)}{(4\pi)^{2k-1}\|f\|^2}a_f(\ell)\sum_{n=1}^\infty A_f(1,n)g(n)=\sideset{}{^h}\sum_{f\in B_{2k}} a_f(\ell)\sum_{n=1}^\infty A_f(1,n)g(n),
\]
where $\sideset{}{^h}\sum$ indicates that the sum is weighted by the reciprocal of the Petersson norm. The first main result is the following average Voronoi formula.

\begin{theorem}\label{thm:holo-voronoi}
Let $\mu=(1-2k,2k-1,0)$ and let $\mathcal{T}_\mu$ be the transform defined in \eqref{equ:defTmu}. For every $g\in\mathcal {S}$ and every $\ell\geqslant 1$,
\begin{equation}\label{equ:holo-voronoi}
I_\ell(g)=I_\ell(\mathcal{T}_\mu g).
\end{equation}
\end{theorem}

From this we obtain the functional equation and entireness of the symmetric-square $L$-functions.

\begin{theorem}\label{thm:holo-fe}
For $f\in B_{2k}$ and $\operatorname{Re}(s)>1$ define
\[
L(s,\operatorname{Sym}^2 f)=\sum_{n=1}^\infty \frac{A_f(1,n)}{n^s}
   =\zeta(2s)\sum_{m=1}^\infty \frac{a_f(m^2)}{m^s}.
\]
Then $L(s,\operatorname{Sym}^2 f)$ extends meromorphically to $\mathbb{C}$ and satisfies the functional equation
\[
L(s,\operatorname{Sym}^2 f)=\mathcal{G}_\mu(1-s)L(1-s,\operatorname{Sym}^2 f),
\]
where 
\begin{equation}\label{equ:def-Gmu}
    \mathcal{G}_\mu(s)=G_0(s+2k-1)G_0(s-2k+1)G_0(s),
\end{equation}
with
\[
    G_0(s)=2(2\pi)^{-s}\Gamma(s)\cos\left(\frac{\pi s}{2}\right)=\frac{\Gamma_{\mathbb{R}}(s)}{\Gamma_{\mathbb{R}}(1-s)}\quad \text{and}\quad\Gamma_{\mathbb{R}}(s)=\pi^{-s/2}\Gamma\left(\frac{s}{2}\right).
\]
Equivalently, with $\Lambda(s,\operatorname{Sym}^2 f)=\Gamma_{\mathbb{C}}(s+2k-1)\Gamma_{\mathbb{R}}(s+1)L(s,\operatorname{Sym}^2 f)$ we have
\[
\Lambda(s,\operatorname{Sym}^2 f)=\Lambda(1-s,\operatorname{Sym}^2 f).
\]
In particular the root number is $+1$.
\end{theorem}

\begin{theorem}\label{thm:holo-entire}
For every $f\in B_{2k}$ the function $L(s,\operatorname{Sym}^2 f)$ is entire. 
\end{theorem}

\subsubsection{Even Hecke–Maass cusp forms}

Let $\mathcal{B}$ be an orthonormal basis of even Hecke–Maass cusp forms of level $1$ and trivial central character. For $u\in\mathcal{B}$ with spectral parameter $t_u$ (so that $\Delta u=(\frac14+t_u^2)u$) write the Fourier expansion as
\[
u(z)=\sqrt{y}\sum_{n\neq0}\rho_u(n)K_{it_u}(2\pi|n|y)e(nx),\qquad
\rho_u(n)=\rho_u(1)\lambda_u(n),\ \lambda_u(1)=1,
\]
and set $\omega_u=4\pi|\rho_u(1)|^2/\cosh(\pi t_u)$. The symmetric-square coefficients are
\[
A_u(1,n)=A_u(n,1)=\sum_{d^2\mid n}\lambda_u\bigl((n/d^2)^2\bigr).
\]
For the continuous spectrum we likewise define $\tau_t(n)=\sum_{ab=n}(a/b)^{it}$ and
\[
A_t(1,n)=A_t(n,1)=\sum_{d^2\mid n}\tau_t\bigl((n/d^2)^2\bigr).
\]

Let $h$ be an even function of complex variable. It is called \emph{Kuznetsov-admissible} if it is holomorphic in $|\mathrm{Im}(t)|<1+\varepsilon$ for soem $\varepsilon>0$ and satisfies $h(s)\ll |s|^{-A}$ for every $A>0$ as $|s|\to\infty$. In our paper, we also acquire $h(\pm i/2)=0$. 

Theorem \ref{thm:maass-voronoi} gives the following Maass‑form analogue of Theorem \ref{thm:holo-voronoi}.

\begin{theorem}\label{thm:maass-voronoi} Let $h$ be a Kuznetsov-admissible function such that $h(\pm i/2)=0$. Let $g\in \mathcal{S}(0,\infty)$. For $\ell\geqslant 1$, define
\begin{equation*}
\begin{split}
I_{\ell}(g;h):=&\sum_{u\in\mathcal{B}}\omega_u h(t_u)\lambda_u(\ell)\sum_{n=1}^\infty A_u(1,n)g(n)
   +\frac1{\pi}\int_{-\infty}^{\infty}\frac{h(t)\tau_t(\ell)}{|\zeta(1+2it)|^2}
   \sum_{n=1}^\infty A_t(1,n)g(n)dt, \\
I_{\ell}^{\vee}(g;h):=&\sum_{u\in\mathcal{B}}\omega_u h(t_u)\lambda_u(\ell)\sum_{n=1}^\infty
   A_u(n,1)\mathcal{T}_{\mu_u}(n)+\frac1{\pi}\int_{-\infty}^{\infty}\frac{h(t)\tau_t(\ell)}{|\zeta(1+2it)|^2}
   \sum_{n=1}^\infty A_t(n,1)\mathcal{T}_{\mu_t}(n)dt,
\end{split}
\end{equation*}
where $\mu_t=(2it,0,-2it)$ with $\mathcal{T}_{\mu_t}$ defined by \ref{def:T-mu-t}. Then
\begin{equation}\label{equ:maass-voronoi}
I_{\ell}(g;h)-I_{\ell}^{\vee}(g;h)=\frac1{\pi}\int_\R h(t)\tau_t(\ell)
\Bigg\{\widetilde g(1-2it)\frac{\zeta(1-4it)}{\zeta(1+2it)}
+\widetilde g(1)
+\widetilde g(1+2it)\frac{\zeta(1+4it)}{\zeta(1-2it)}
\Bigg\}dt.
\end{equation}
\end{theorem}

From this we deduce the functional equation and entireness of the associated symmetric-square $L$-functions.

\begin{theorem}\label{thm:maass-fe}
For $u\in\mathcal{B}$ and $\operatorname{Re}(s)>1$ set
\[
L(s,\operatorname{Sym}^2 u)=\sum_{n=1}^\infty \frac{A_u(1,n)}{n^s}
   =\zeta(2s)\sum_{m=1}^\infty \frac{\lambda_u(m^2)}{m^s}.
\]
Then $L(s,\operatorname{Sym}^2 u)$ extends meromorphically to $\mathbb{C}$ and satisfies
\[
L(s,\operatorname{Sym}^2 u)=\mathcal{G}_{\mu_u}(1-s)L(1-s,\operatorname{Sym}^2 u),
\]
where $\mu_u=(2it_u,0,-2it_u)$ and
\[
\mathcal{G}_{\mu_u}(s)=G_0(s+2it_u)G_0(s)G_0(s-2it_u)=\frac{\Gamma_{\mathbb{R}}(s+2it_u)\Gamma_{\mathbb{R}}(s)\Gamma_{\mathbb{R}}(s-2it_u)}
                            {\Gamma_{\mathbb{R}}(1-s-2it_u)\Gamma_{\mathbb{R}}(1-s)\Gamma_{\mathbb{R}}(1-s+2it_u)}.
\]
Equivalently, with
\[
\Lambda(s,\operatorname{Sym}^2 u)=\Gamma_{\mathbb{R}}(s+2it_u)\Gamma_{\mathbb{R}}(s)\Gamma_{\mathbb{R}}(s-2it_u)L(s,\operatorname{Sym}^2 u)
\]
we have $\Lambda(s,\operatorname{Sym}^2 u)=\Lambda(1-s,\operatorname{Sym}^2 u)$. In particular the root number is $+1$.
\end{theorem}

\begin{theorem}\label{thm:maass-entire}
For every $u\in\mathcal{B}$, the function $L(s,\operatorname{Sym}^2 u)$ is entire.
\end{theorem}

\subsection{Outline of the method}

We prove averaged Voronoi identities for symmetric-square coefficients via the following path. (For simplicity, we only outline the holomorphic case.)

\begin{enumerate}
    \item \textbf{Trace formula expansion.}
    Express $I_\ell(g)$ using the Petersson trace formula (holomorphic case) or the Kuznetsov trace formula (Maass case), decomposing it into a diagonal term $\Delta_\ell(g)$ and a Kloosterman sum term $\mathcal{G}_\ell(g)$.

    \item \textbf{Poisson summation.}
    Apply Poisson summation to $\mathcal{G}_\ell(g)$. This yields
    \[
        \mathcal{G}_\ell(g)=\frac{\pi i^{-2k}}{2\pi i}\int_{(\sigma)}
        \widetilde{g}(1-z)
        \sum_{n\in\mathbb{Z}}
        \mathscr{L}_{n^2-4\ell}(1-z)\,\mathcal{J}_{n,\ell}(z)\,dz.
    \]

    \item \textbf{Treat $I_{\ell}(\mathcal{T}_{\mu}g)$ similarly.}
    Repeating the first two steps for $I_{\ell}(\mathcal{T}_{\mu}g)$ gives
    $I_{\ell}(\mathcal{T}_{\mu}g)=\Delta_\ell(\mathcal{T}_{\mu}g)+\mathcal{G}_\ell(\mathcal{T}_{\mu}g)$ with
    \[
        \mathcal{G}_\ell(\mathcal{T}_{\mu}g)=\frac{\pi i^{-2k}}{2\pi i}\int_{(\sigma)}
        \widetilde{\mathcal{T}_{\mu}g}(1-z)
        \sum_{n\in\mathbb{Z}}
        \mathscr{L}_{n^2-4\ell}(1-z)\,\mathcal{J}_{n,\ell}(z)\,dz.
    \]

    \item \textbf{Archimedean reciprocity.}
    For all $n$ with $n^2 \neq 4\ell$, the contributions to $\mathcal{G}_\ell(g)$ and $\mathcal{G}_\ell(\mathcal{T}_{\mu}g)$ coincide. This follows from a functional equation relating $\mathscr{L}_{n^2-4\ell}(1-z)\mathcal{J}_{n,\ell}(z)$ to its counterpart with $z$ replaced by $1-z$, combined with the definition of $\mathcal{T}_\mu$.

    \item \textbf{Singular and diagonal corrections.}
    When $\ell = m^2$ is a perfect square, the remaining terms satisfy
    \[
        \sum_{n=\pm 2m}\bigl(\text{contribution from }\mathcal{G}_\ell(g)\bigr) \;+\; \Delta_\ell(g)
        \;=\;
        \sum_{n=\pm 2m}\bigl(\text{contribution from }\mathcal{G}_\ell(\mathcal{T}_{\mu}g)\bigr) \;+\; \Delta_\ell(\mathcal{T}_{\mu}g).
    \]
    This equality is verified by a direct residue calculation using the explicit forms of $\mathscr{L}_0(s)=\zeta(2s-1)$ and $\mathcal{J}_{\pm2m,m^2}(z)$.
\end{enumerate}

The Maass case follows the same plan, with the Kuznetsov trace formula replacing the Petersson trace formula, the kernel $\mathbb{J}_t$ replacing $J_{2k-1}$, and the spectral-parameter-dependent factor $\mathcal{G}_{\mu_t}$ in place of $\mathcal{G}_\mu$. The continuous spectrum contributes an extra explicit term in Theorem \ref{thm:maass-voronoi}.

\subsection*{Acknowledgements}
\addtocontents{toc}{\protect\setcounter{tocdepth}{1}}
T. Deng is supported by Beijing Natural Science Foundation, No. 1244042 and National Natural Science Foundation of
China, No. 12401013.

\section{Preliminaries}

This section collects the necessary background for the proof, including special functions (Gamma, Bessel, hypergeometric, Weber–Schafheitlin integrals), trace formulas (Poisson summation, Petersson, Kuznetsov), arithmetic sums (Kloosterman sums, Ramanujan sums), and the quadratic Dirichlet series $\mathscr{L}_D(s)$ with its meromorphic continuation and functional equation.

\subsection{Special functions}
\subsubsection{Gamma function}
The Gamma function $\Gamma(s)$ is defined by
\[
    \Gamma(s)=\int_0^{\infty} e^{-x}x^{s-1}dx,\quad\mathrm{Re}(s)>0.
\]
A partial integration gives the recursion formula $\Gamma(s+1)=s\Gamma(s)$. This formula yields a meromorphic extension of $\Gamma(s)$ to $\mathbb{C}$, with simple poles at non-positive integers $0,-1,-2,\dots$.

We need the following properties of $\Gamma(s)$.

\begin{itemize}
    \item \emph{Functional equation.}
    \[
    \Gamma(s)\Gamma(1-s)=\frac{\pi s}{\sin(\pi s)}.    
    \]
    \item \emph{Legendre duplication formula.}
    \[
    \Gamma(s)=\pi^{-\frac12}2^{s-1}\Gamma\big(\frac{s}{2}\big)\Gamma\big(\frac{1+s}{2}\big).
    \]
    \item \emph{Stirling's formula.}\quad In the angle $|\mathrm{arg}(s)|<\pi-\epsilon$, we have the estimate
    \[
        \Gamma(s)=\left(\frac{2\pi}{s}\right)^{1/2}\left(\frac{s}{e}\right)^s(1+O_{\epsilon}(|s|^{-1})).
    \]
    A more useful version is that, let $a<b$ be two real numbers, then in the vertical strip $a\leqslant\sigma=\mathrm{Re}(s)\leqslant b$, and $t=\mathrm{Im}(s)>0$, one has
    \[
        \Gamma(\sigma+it)=(2\pi)^{1/2}t^{\sigma-1/2}e^{-\pi t/2}\left(\frac{t}{e}\right)^{it}(1+O_{a,b}(t^{-1}))\quad\text{as}~~t\to\infty.
    \]
\end{itemize}
We shall repeatedly use Euler's reflection formula, Legendre's duplication formula and Stirling's formula without further reference.

\subsubsection{Bessel functions}
For $\nu\in\mathbb{C}$ and $x>0$, the \emph{Bessel function of the first kind} is defined by the series
\[
J_{\nu}(x)=\sum_{m=0}^{\infty}\frac{(-1)^m}{m!\Gamma(m+\nu+1)}\Bigl(\frac{x}{2}\Bigr)^{2m+\nu}.
\]
When $x\to 0^+$, we have $J_{\nu}^{(j)}(x)\ll_{\nu} x^{\nu-j}$ if $\nu\notin\mathbb{Z}_{\geqslant 0}$ and $J_{\nu}^{(j)}\ll_{\nu} x^{\max(\nu-j, 0)}$ if $\nu\in\mathbb{Z}_{\geqslant 0}$. While as $x\to+\infty$, the Bessel function has the following asymptotic expansion 
\[
J_{\nu}(x)=\sqrt{\frac{2}{\pi x}}\Bigl(\cos\bigl(x-\tfrac{\nu\pi}{2}-\tfrac{\pi}{4}\bigr)+O\bigl(\tfrac{1+|\nu|^2}{x}\bigr)\Bigr).
\]

If we fix $x>0$, then $J_{\nu}(x)$ is an entire function of the parameter $\nu$. Moreover, if $\mathrm{Re}(\nu)$ is a half integer, then
\begin{equation*}
   J_{\nu}(x)\ll_x \frac{1}{|\Gamma(\nu+1)|}\left(\frac{x}{2}\right)^{\mathrm{Re}(\nu)}.
\end{equation*}

For our purposes we need the following special case when $\nu=-1/2$:
\[
    J_{-1/2}(x)=\sqrt{\frac{2}{\pi x}}\cos(x).
\]
We also need the modified Bessel combination that appears in the Kuznetsov formula:
\begin{equation}\label{equ:mathbbJ}
\mathbb{J}_t(x)=\frac{i}{\pi}\frac{t}{\cosh(\pi t)}\bigl(J_{2it}(x)-J_{-2it}(x)\bigr),\qquad t\in\mathbb{R}.
\end{equation}
This function is even in $t$ and real for $x>0$.

For $-\mathrm{Re}(\nu)<\mathrm{Re}(z)<\frac{1}{2}$, we have
\begin{equation}\label{equ:mellin-Bessel}
    \int_0^{\infty}x^{z-1}J_{\nu}(x)dx=2^{z-1}\frac{\Gamma\big(\frac{\nu+z}{2}\big)}{\Gamma\big(1+\frac{\nu-z}{2}\big)}.
\end{equation}
Here, the integral converges absolutely when $-\mathrm{Re}(\nu)<\mathrm{Re}(z)<\frac{1}{2}$.

For a detailed treatment for Bessel functions, see \cite{Wat95}.

\subsubsection{Hypergeometric functions}
For $a,b,c\in\mathbb{C}$ with $c\notin\{0,-1,-2,\dots\}$ and $|x|<1$, the Gauss hypergeometric function ${}_2F_1(a,b;c;x)$ is defined by the series
\[
{}_2F_1(a,b;c;x)=\sum_{n=0}^{\infty}\frac{(a)_n(b)_n}{(c)_n n!}x^n,
\]
where $(a)_n=\Gamma(a+n)/\Gamma(a)$ is the Pochhammer symbol. The series converges absolutely for $|x|<1$ and defines a holomorphic function there. 

We use the following standard formulas for Gauss hypergeometric functions (see DLMF \S 15).
\begin{itemize}
    \item The \emph{Euler's transformation} reads as
    \[
     {}_2F_1(a,b;c;x) = (1-x)^{c-a-b}{}_2F_1(c-a,c-b;c;x),\qquad |\arg(1-x)|<\pi.
    \]
    \item The \emph{Euler integral representation} says that
      \begin{equation}\label{equ:Euler-integral-rep}
        {}_2F_1(a,b;c;x)=\frac{\Gamma(c)}{\Gamma(b)\Gamma(c-b)}\int_0^1 t^{b-1}(1-t)^{c-b-1}(1-xt)^{-a}dt,
\end{equation}
holds under the additional conditions $\operatorname{Re}(c)>\operatorname{Re}(b)>0$ and $|\arg(1-x)|<\pi$. Under the same conditions, we have the estimate
\begin{equation}\label{equ:bound-hypergeo}
\left|{}_2F_1(a,b;c;x)\right|\leqslant \left|\frac{\Gamma(c)}{\Gamma(b)\Gamma(c-b)}\right| \int_0^1 t^{\mathrm{Re}(b)-1}(1-t)^{\mathrm{Re}(c-b)-1}(1-xt)^{-\mathrm{Re}(a)}dt.
\end{equation}
\item Another formula we need is
\begin{equation}\label{equ:hyper-geo-square-root}
    {}_2F_1\left(\frac{\nu}{2},\frac{\nu+1}{2};1+\nu;x\right)
=\left(\frac{2}{1+\sqrt{1-x}}\right)^\nu.
\end{equation}
\end{itemize}

\subsubsection{The Weber–Schafheitlin integral}

The following evaluation is a standard Weber–Schafheitlin integral (see Watson \cite[p. 401 and p. 403]{Wat95} or DLMF \S 10.22), which gives the Mellin transform of a product of two Bessel functions.

\begin{lemma}[Weber–Schafheitlin]\label{lem:WS}
Let $\alpha,\beta,\lambda\in\mathbb{C}$ satisfy $0<\operatorname{Re}(\lambda)<\operatorname{Re}(\alpha+\beta+1)$. 

If $0<a<b$, we have
\begin{equation}\label{equ:WS-neq}
\begin{split}
&\int_{0}^{\infty}t^{-\lambda}J_{\alpha}(at)J_{\beta}(bt)dt\\
=& \frac{a^{\alpha}\Gamma\bigl(\frac{\alpha+\beta-\lambda+1}{2}\bigr)}
        {2^{\lambda}b^{\alpha-\lambda+1}\Gamma(\alpha+1)\Gamma\bigl(\frac{\beta-\alpha+\lambda+1}{2}\bigr)}
    {}_2F_1\left(\frac{\alpha+\beta-\lambda+1}{2},\frac{\alpha-\beta-\lambda+1}{2};\alpha+1;\frac{a^{2}}{b^{2}}\right).
\end{split}
\end{equation} 
In the equal argument case $a=b>0$ one has
\begin{equation}\label{equ:WS-eq}
\int_{0}^{\infty}t^{-\lambda}J_{\alpha}(at)J_{\beta}(at)dt
= \frac{a^{\lambda-1}\Gamma(\lambda)\Gamma\bigl(\frac{\alpha+\beta-\lambda+1}{2}\bigr)}
        {2^{\lambda}\Gamma\bigl(\frac{-\alpha+\beta+\lambda+1}{2}\bigr)
                \Gamma\bigl(\frac{\alpha+\beta+\lambda+1}{2}\bigr)
                \Gamma\bigl(\frac{\alpha-\beta+\lambda+1}{2}\bigr)}.
\end{equation}
\end{lemma}

\subsection{Trace formulas}

\subsubsection{Poisson summation formula}

\begin{lemma}[Poisson summation]\label{lem:poisson-summation}
Let $f$ be a continuous function on $\mathbb{R}$ such that $f$, $f'$, and $f''$ belong to $L^1(\mathbb{R})$. Then
\begin{equation}\label{equ:poisson-summation}
  \sum_{n\in\mathbb{Z}}f(n)=\sum_{n\in\mathbb{Z}}\widehat f(n), 
\end{equation}
where the Fourier transform is defined by
\[
\widehat{f}(y) = \int_{-\infty}^{\infty} f(x) e^{-2\pi i x y} dx.
\]
\end{lemma}

\subsubsection{Petersson trace formula of level \texorpdfstring{$1$}{1}}

Let $S_{2k}$ be the space of holomorphic cusp forms of weight $2k$ and level $1$, and let $\{f_j\}$ be a basis of Hecke eigenforms with Fourier expansions
\[
f_j(z)=\sum_{n\geqslant1} a_j(n)n^{\frac{2k-1}{2}} e^{2\pi i n z},\qquad a_j(1)=1.
\]
The \emph{Petersson trace formula} states that for any positive integers $m,n$,
\[
\sum_{j}\frac{\Gamma(2k-1)}{(4\pi)^{2k-1}\|f_j\|^2}a_j(m)\overline{a_j(n)}
= \delta_{m,n}+2\pi i^{-2k}\sum_{c\geqslant1}\frac{S(m,n;c)}{c}J_{2k-1}\left(\frac{4\pi\sqrt{mn}}{c}\right).
\]
Here $\|f_j\|$ is the Petersson inner product of $f_j$, and $J_{2k-1}$ is the Bessel function of the first kind.

\subsubsection{Kuznetsov trace formula (same-sign)}

We shall use the same-sign Kuznetsov trace formula in the following normalization.
Let $h$ be a Kuznetsov-admissible function. For such
$h$, put
\[
H_0(h):=\frac1{\pi^2}\int_{-\infty}^{\infty}
t\tanh(\pi t)h(t)dt
\]
and
\[
\mathbb J_t(x):=
\frac{i}{\pi}\frac{t}{\cosh(\pi t)}
\bigl(J_{2it}(x)-J_{-2it}(x)\bigr),
\qquad
\mathcal J_h(x):=\int_{-\infty}^{\infty}h(t)\mathbb J_t(x)dt .
\]

Then for positive integers $m,n$ (cf. \cite[Theorem 7.14]{KL13}),
\begin{equation}\label{equ:level-one-kuznetsov}
\begin{aligned}
&\sum_{u\in\mathcal B}\omega_u h(t_u)\lambda_u(m)\overline{\lambda_u(n)}
+\frac1{\pi}\int_{-\infty}^{\infty}
\frac{h(t)\tau_t(m)\overline{\tau_t(n)}}{|\zeta(1+2it)|^2}dt  \\
&\qquad =
\delta_{m,n}H_0(h)+
\sum_{c\geqslant1}\frac{S(m,n;c)}{c}
\mathcal J_h\left(\frac{4\pi\sqrt{mn}}{c}\right).
\end{aligned}
\end{equation}
For level $1$ and trivial central character the Hecke eigenvalues
$\lambda_u(n)$ and the divisor coefficients $\tau_t(n)$ are real, so the
conjugates are harmless; they are retained to make the sesquilinear
normalization of Kuznetsov explicit.

\subsection{Arithmetic sums}
\subsubsection{Kloosterman sums}
For positive integers $c$ and arbitrary integers $m,n$, the Kloosterman sum is defined by
\[
S(m,n;c)=\sum_{x\bmod c,~~(x,c)=1} e\left(\frac{mx+n\bar{x}}{c}\right),
\]
where $\bar{x}$ denotes the multiplicative inverse of $x$ modulo $c$ and $e(z)=e^{2\pi iz}$. These sums satisfy the Weil bound
\[
|S(m,n;c)|\le \tau(c)\sqrt{c}\gcd(m,n,c)^{1/2},
\]
where $\tau(c)$ is the divisor function. In particular, for fixed $m,n$ one has $S(m,n;c)=O(c^{1/2+o(1)})$ as $c\to\infty$.

\subsubsection{Ramanujan sums}
The Ramanujan sum is defined for an integer $q\geqslant 1$ by
\[
c_q(n)=\sum_{a\bmod q,~~(a,q)=1} e\left(\frac{an}{q}\right).
\]
Using the Möbius inversion one obtains the explicit formula
\[
c_q(n)=\sum_{d\mid (q,n)} d\mu\left(\frac{q}{d}\right).
\]

A crucial arithmetic identity (used in the evaluation of $\Sigma_c(\ell,n)$) is
\begin{equation}\label{equ:Ramanujansum}
\sideset{}{^*}\sum_{x\bmod q} e\left(\frac{\ell\bar{x}}{q}\right)\sum_{a\bmod q} e\left(\frac{xa^{2}+an}{q}\right)
= q\sum_{d\mid q}\mu\left(\frac{q}{d}\right)\rho_d(n,\ell,2k-1),
\end{equation}
where
\begin{equation}\label{equ:def-rhodnl}
\rho_d(n,\ell,2k-1)=\#\{u\bmod d: u^{2}+nu+\ell\equiv0\pmod{d}\}.
\end{equation}
\emph{Derivation.} Let $u=ax$. Then the left hand side is equal to
\begin{equation*}
\sideset{}{^\ast}\sum_{x\bmod{q}}e\left(\frac{\ell\overline{x}}{q}\right)\sum_{u\bmod{q}}
e\left(\frac{u^2\overline{x}+nu\overline{x}}{q}\right)=\sum_{u\bmod{q}}~~\left[~~\sideset{}{^\ast}\sum_{x\bmod{q}}e\left(\frac{(u^2+nu+\ell)\overline{x}}{q}\right)\right]=\sum_{u\bmod{q}}c_q(u^2+nu+\ell).
\end{equation*}
Thus the desired arithmetic identity \eqref{equ:Ramanujansum} follows.

\subsection{Quadratic Dirichlet series \texorpdfstring{$\mathscr{L}_D(s)$}{L\_D(s)}}

Let $D$ be an integer. Following Zagier \cite{Zag06}, we write
\[
\mathscr{L}_D(s)=
    \frac{\zeta(2s)}{\zeta(s)}\sum_{r\geqslant 1}\frac{\rho_r(D)}{r^s},
\]
where $\rho_r(D)=\#\{u\bmod r : u^2\equiv D\pmod{r}\}=\rho_{r}(n,\ell,2k-1)$ with $D=n^2-4\ell$. (see \eqref{equ:def-rhodnl}). These coefficients appear naturally after applying Poisson summation to quadratic exponential sums. For $D=0$ one has $\mathscr{L}_0(s)=\zeta(2s-1)$.

The series $\mathscr{L}_D(s)$ can be expressed in terms of Dirichlet $L$-functions attached to quadratic characters. Write $D=D_0f^2$ where $D_0$ is a fundamental discriminant (with $D_0=1$ when $D$ is a perfect square). Then Zagier's formula (see \cite{Zag06}) gives
\begin{equation}\label{equ:Zagier}
\mathscr{L}_D(s)=L(s,\chi_{D_0})\sum_{a\mid f}\mu(a)\chi_{D_0}(a)a^{-s}\sigma_{1-2s}\left(\frac{f}{a}\right),
\end{equation}
where $\sigma_w(n)=\sum_{d\mid n}d^w$. In particular, $\mathscr{L}_D(s)$ is a finite linear combination of Dirichlet $L$-series and therefore admits a meromorphic continuation to $\mathbb{C}$ with at most a simple pole at $s=1$ (when $D$ is a square, coming from the Riemann zeta factor). The functional equation is
\begin{equation}\label{equ:LD_fe}
  \mathscr{L}_D(1-s)=\Xi_D(s)\mathscr{L}_D(s), \qquad \Xi_D(s)=\left(\frac{|D|}{\pi}\right)^{s-\frac12}
\frac{\Gamma\left(\frac{s+\delta_D}{2}\right)}{\Gamma\left(\frac{1-s+\delta_D}{2}\right)},  
\end{equation}
with $\delta_D=0$ if $D>0$ and $\delta_D=1$ if $D<0$. This is a standard consequence of the functional equation of Dirichlet $L$-functions and the properties of the divisor sum $\sigma_{1-2s}$; see \cite{Zag06} for a detailed derivation.

For our purposes we only need the following uniform bound (convexity bound) on vertical strips: for any fixed $a<b$ and any constant $\delta>0$, there exists a constant $A=A(a,b)>0$ such that
\begin{equation}\label{equ:convexbound}
 \mathscr{L}_D(s)\ll_{a,b,\delta} (1+| \mathrm{Im}(s)|)^{A}|D|^{\max(0,(1-\operatorname{Re}s)/2)+\delta},   
\end{equation}
for $a<\mathrm{Re}(s)<b$ apart from the possible simple pole at $s=1$ when $D$ is a square. The bound follows from the corresponding estimate for the Dirichlet $L$-function and the divisor sum $\sigma_{1-2s}(f/a)$, which grows at most polynomially in $f$ and in $| \mathrm{Im}(s)|$.

When $D$ is a non‑zero square, the factor $L(s,\chi_{D_0})$ becomes $\zeta(s)$ (up to a finite Euler product), and $\mathscr{L}_D(s)$ has a simple pole at $s=1$. Moreover, by \eqref{equ:Zagier}, we have
\begin{equation*}
    \operatorname*{Res}_{s=1}\mathscr{L}_{d^2}(s)=\begin{cases}
        \frac12,&\quad d=0,\\
        1,&\quad d\geqslant 1.
    \end{cases}
\end{equation*}

\section{The Archimedean Kernels and Their Transforms}

This section defines the Schwartz space $\mathcal{S}$ and its Mellin transform, then introduces the key archimedean transforms $\mathcal{T}_\mu$ (holomorphic case) and $\mathcal{T}_{\mu_t}$ (Maass case), studying their growth properties. Explicit formulas, functional equations, and singularity behaviors of the kernels $\mathcal{J}_{n,\ell}(z)$ and $\mathcal{J}_{n,\ell,t}(z)$ are derived in detail to prepare for the averaged Voronoi identities.

\subsection{Schwartz space and Mellin transform}
\begin{definition}\label{def:schwartz}
The space $\mathcal{S}=\mathcal{S}(0,\infty)$ consists of all smooth functions $f:(0,\infty)\to\mathbb{C}$ such that for every $m\in\mathbb{Z}$ and every $n\in\mathbb{Z}_{\geqslant 0}$
\[
\sup_{x>0}\bigl|x^m f^{(n)}(x)\bigr|<\infty .
\]
Equivalently, $f$ decays rapidly as $x\to\infty$ and together with all its derivatives vanishes rapidly as $x\to 0^+$.
\end{definition}

For an integrable function $f$ on $(0,\infty)$ the Mellin transform is defined by
\[
\widetilde{f}(s)=\int_0^\infty f(x)x^{s-1}dx 
\]
whenever the integral converges absolutely. 

If $f\in\mathcal{S}$, then $\widetilde{f}(s)$ is entire and satisfies the following property: for any real numbers $a<b$ and any integer $m\ge0$,
\[
\sup_{a\le\operatorname{Re}(s)\le b}\bigl|s^m\widetilde{f}(s)\bigr|<\infty . 
\]
We say $\widetilde{f}(s)$ is \emph{essentially bounded} if the above estimate holds for every $m\in\mathbb{Z}_{\geqslant 0}$ and any real numbers $a<b$.

A useful observation is that the converse also holds. Namely,
\begin{lemma}\label{lem:mellin-Schwartz}
Let $F(s)$ be an entire function. If $F$ is essentially bounded, then the Mellin inversion
\begin{equation*}
    f(x)=\frac{1}{2\pi i}\int_{(\sigma)}F(s)x^{-s}ds=\frac{1}{2\pi}\int_{-\infty}^{\infty}F(\sigma+it)x^{-\sigma-it}dt.
\end{equation*}
is a function in $\mathcal {S}$, where $\sigma\in\mathbb{R}$.
\end{lemma}

\begin{proof}
Since $F$ is entire and essentially bounded, the Mellin inversion can be differentiated under the integral sign for any $\sigma\in\mathbb{R}$:
\[
f^{(m)}(x) = \frac{(-1)^m}{2\pi i} \int_{(\sigma)} F_m(s) x^{-s-m} ds,
\]
where $F_m(s)=s(s+1)\cdots(s+m-1)F(s)$ is also entire and essentially bounded.

Let $A>0$. For $x\to\infty$, taking $\sigma=A-m$  yields $|f^{(m)}(x)|\ll x^{-A}$; for $x\to0^+$, taking $\sigma=-A+m$ yields $|f^{(m)}(x)|\ll x^{A}$.
\end{proof}

\subsection{The holomorphic case}

\subsubsection{The archimedean transform \texorpdfstring{$\mathcal{T}_\mu$}{T\_mu}}

For the holomorphic case we fix an integer $k\geqslant 2$ and set
\[
\mu=(1-2k,2k-1,0).
\]
\begin{definition}\label{def:T_mu_holo}
For $g\in\mathcal{S}$ define $\mathcal{T}_\mu g:(0,\infty)\to\mathbb{C}$ by Mellin inversion
\begin{equation}\label{equ:defTmu}
(\mathcal{T}_\mu g)(y)=\frac1{2\pi i}\int_{(\sigma)} \widetilde{g}(1-s)\mathcal{G}_\mu(s)y^{-s}ds,
\qquad \sigma>0,
\end{equation}
where $\mathcal{G}_\mu(s)$ is defined in \eqref{equ:def-Gmu}. Equivalently, the Mellin transform of $\mathcal{T}_\mu g$ is
\[
\widetilde{\mathcal{T}_\mu g}(s)=\widetilde{g}(1-s)\mathcal{G}_\mu(s).
\]
\end{definition}

Recall that
\begin{equation*}
    \mathcal{G}_\mu(s)=G_0(s+2k-1)G_0(s-2k+1)G_0(s),
\end{equation*}
where
\[
G_0(s)=2(2\pi)^{-s}\Gamma(s)\cos\left(\frac{\pi s}{2}\right)=\frac{\Gamma_{\mathbb{R}(s)}}{\Gamma_{\mathbb{R}}(1-s)}.
\]
So the function $\mathcal{G}_\mu(s)$ is holomorphic when $\operatorname{Re}(s)>0$ and satisfies the functional equation
\[
\mathcal{G}_\mu(s)\mathcal{G}_\mu(1-s)=1.
\]
Moreover, by Stirling's formula, $|\mathcal{G}_\mu(s)|$ grows at most polynomially in $|\operatorname{Im}(s)|$ on any finite vertical strip.

The transform $\mathcal{T}_\mu g$ enjoys the following properties, which are weaker than those of $g$ but sufficient for our purposes.

\begin{lemma}\label{lem:Tmug-props}
Let $g\in \mathcal{S}$. Then $\mathcal{T}_\mu g$ is smooth on $(0,\infty)$. Moreover:
\begin{itemize}
    \item For every $A>0$ and every $m\ge 0$,
    \[
    (\mathcal{T}_\mu g)^{(m)}(y) \ll y^{-A} \quad \text{as } y\to\infty.
    \]
    \item As $y\to 0^+$,
    \[
    \mathcal{T}_\mu g(y)=O(y),\quad (\mathcal{T}_\mu g)'(y)=O(1),\quad (\mathcal{T}_\mu g)''(y)=O(1).
    \]
\end{itemize}
The Mellin transform $\widetilde{\mathcal{T}_\mu g}(s)$ is holomorphic for $\operatorname{Re}(s)>0$. Furthermore, for any $0<a<b$ and any $N>0$,
\[
\sup_{a\le \operatorname{Re}(s)\le b} \bigl| s^N \widetilde{\mathcal{T}_\mu g}(s) \bigr| < \infty.
\]
\end{lemma}

\begin{proof}
Since $g\in\mathcal{S}$, $\widetilde{g}$ is entire and essentially bounded. The function
\[
\widetilde{\mathcal{T}_\mu g}(s)=\widetilde{g}(1-s)\mathcal{G}_\mu(s)
\]
inherits from $\widetilde{g}$ and $\mathcal{G}_\mu$ the following properties:
\begin{itemize}
    \item it is holomorphic for $\operatorname{Re}(s)>0$;
    \item it is polynomially bounded on vertical strips.
\end{itemize}

Let $A>0$. For $y\to\infty$, choose $\sigma = A+m$. Then
\[
(\mathcal{T}_\mu g)^{(m)}(y) = \frac{(-1)^m}{2\pi i} \int_{(A+m)} s\cdots(s+m-1) \widetilde{\mathcal{T}_\mu g}(s) \, y^{-s-m} ds,
\]
and the integrand decays rapidly, yielding $(\mathcal{T}_\mu g)^{(m)}(y) \ll y^{-A-m}$.

For $y\to0^+$, shift the contour to $\operatorname{Re}(s)=-\frac52$ and collect residues at $s=-1,-2$:
\[
(\mathcal{T}_\mu g)^{(m)}(y)=(-1)^m\frac{1}{2\pi i}\int_{(-\frac52)} s\cdots(s+m-1)\widetilde{\mathcal{T}_\mu g}(s)y^{-s-m} ds + \operatorname*{Res}_{s=-1} + \operatorname*{Res}_{s=-2}.
\]
The contour integral contributes lower-order terms, while the residues determine the leading asymptotic:
\[
\mathcal{T}_\mu g(y)=O(y),\quad (\mathcal{T}_\mu g)'(y)=O(1),\quad (\mathcal{T}_\mu g)''(y)=O(1) \qquad (y\to0^+).
\]
\end{proof}

\subsubsection{The holomorphic kernel \texorpdfstring{$\mathcal{J}_{n,\ell}(z)$}{J\_\{n,ℓ\}(z)}}
For every integer $n$, every integer $\ell\geqslant1$ and a complex parameter $z$ with $-1/2<\operatorname{Re}(z)<1/2$ we define
\[
\mathcal{J}_{n,\ell}(z)=2\int_{0}^{\infty}y^{z-1}J_{2k-1}(4\pi\sqrt{\ell}y)\cos(2\pi ny)dy.
\]
More generally, for $\nu\in\mathbb{C}$ with $\mathrm{Re}>1\frac12$, we define
\begin{equation*}
\mathcal{J}^{\nu}_{n,\ell}(z):=2\int_{0}^{\infty}y^{z-1}J_{\nu}(4\pi\sqrt{\ell}y)\cos(2\pi |n|y)dy,
\end{equation*}
which converges absolutely when $-\mathrm{Re}(\nu)<\mathrm{Re}(z)<\frac12$.

\begin{lemma}\label{lem:J-nu-n-ell}
Let $\nu\in\mathbb{C}$ with $\mathrm{Re}(\nu)\geqslant 0$. Denote
\begin{equation*}
\mathcal{J}^{\nu}_{n,\ell}(z):=2\int_{0}^{\infty}y^{z-1}J_{\nu}(4\pi\sqrt{\ell}y)\cos(2\pi |n|y)dy.
\end{equation*}
The integral converges absolutely when $-\mathrm{Re}(\nu)<\mathrm{Re}(z)<\frac12$. Moreover,
\begin{itemize}
    \item If $n=0$, we have
    \[
        \mathcal{J}^{\nu}_{0,\ell}(z)=2^{-z}\pi^{-z}\ell^{-\frac{z}{2}}\frac{\Gamma(\frac{\nu+z}{2})}{\Gamma(\frac{\nu+1-z}{2})}.
    \]
    \item If $n\neq 0$ and $n^2-4\ell<0$, we have
    \[
        \mathcal{J}^{\nu}_{n,\ell}(z)=C_{-}(z;\ell,\nu)
     {}_2F_1\left(\frac{\nu+z}{2},\frac{-\nu+z}{2};\frac12;\frac{n^{2}}{4\ell}\right)
    \]
    with
     \[
     C_{-}(z;\ell,\nu)=2^{-z}\pi^{-z}\ell^{-z/2}
                   \frac{\Gamma\bigl(\frac{\nu+z}{2}\bigr)}
                        {\Gamma\bigl(\frac{\nu+1-z}{2}\bigr)}.
     \]
    If we take $n=0$, it coincides with the previous case.
    \item If $n^2-4\ell>0$, we have
    \[
     \mathcal{J}^{\nu}_{n,\ell}(z)=C_{+}(z;n,\ell,\nu)
     {}_2F_1\left(\frac{\nu+z}{2},\frac{\nu+z+1}{2};\nu+1;\frac{4\ell}{n^{2}}\right),
     \]
     where
     \[
     C_{+}(z;n,\ell,\nu)=2^{\nu}\pi^{\frac{1}{2}-z}\ell^{\frac{\nu}{2}}|n|^{-\nu-z}
                     \frac{\Gamma\bigl(\frac{\nu+z}{2}\bigr)}
                          {\Gamma(\nu+1)\Gamma\bigl(\frac{-\nu-z+1}{2}\bigr)}.
     \]
     \item If $\ell=m^2$ for some $m\geqslant 1$ and $n^2=4\ell$, then
     \begin{equation*}
     \mathcal{J}^{\nu}_{\pm 2m,m^{2}}(z)=2^{-z}\pi^{\frac{1}{2}-z}m^{-z}
        \frac{\Gamma\bigl(\frac{\nu+z}{2}\bigr)\Gamma\bigl(\frac12-z\bigr)}
             {\Gamma\bigl(\frac{-\nu-z+1}{2}\bigr)\Gamma\bigl(\frac{\nu-z}{2}+1\bigr)\Gamma\bigl(\frac{\nu-z+1}{2}\bigr)}.
     \end{equation*}
\end{itemize}
\end{lemma}

\begin{proof}
The case $n=0$ follows from \eqref{equ:mellin-Bessel}. If $n\neq 0$, using $\cos(x)=\sqrt{\frac{\pi x}{2}}J_{-1/2}(x)$, we obtain
\[
   \mathcal{J}^{\nu}_{n,\ell}(z):=2\pi\sqrt{|n|}\int_{0}^{\infty}y^{z-\frac{1}{2}}J_{\nu}(4\pi\sqrt{\ell}y)J_{-\frac12}(2\pi |n|y)dy.
\]
This is an integral of the Weber–Schafheitlin type. Applying Lemma \ref{lem:WS} yields the desired results. For example, if $n^2-4\ell>0$, then we take
\[
    \alpha=\nu,\quad a=4\pi\sqrt{\ell},\quad \beta=-\frac{1}{2},\quad b=2\pi|n|y\quad\text{and}\quad \lambda=\frac{1}{2}-z.
\]
\end{proof}

From the explicit formulas, we can derive some useful functional equations for $\mathcal{J}_{n,\ell}(z)=\mathcal{J}^{2k-1}_{n,\ell}(z)$.

\begin{lemma}\label{lem:J-2m-m^2-fe}
For $m\geqslant 1$, we have
\begin{equation}\label{equ:J-2m-m^2-fe}
    2\pi i^{-2k}\zeta(1-2z)\mathcal{J}_{\pm 2m,m^2}(z)=\mathcal{G}_{\mu}(z)m^{-z}\zeta(2z).
\end{equation}
\end{lemma}

\begin{proof}
Substitute the explicit formula for $\mathcal{J}_{\pm2m,m^2}(z)$ from Lemma \ref{lem:J-nu-n-ell} and use $\zeta(1-s)=G_0(s)\zeta(s)$ and $\Gamma(z)=\pi^{-1/2}2^{z-1}\Gamma(\frac{z}{2})\Gamma(\frac{z+1}{2})$. The left-hand side becomes
\[
(-1)^k m^{-z} \pi^{\frac32-3z} \zeta(2z) \,
\frac{\Gamma(\frac{z}{2})\Gamma(\frac{z+1}{2})\Gamma(k-\frac12+\frac{z}{2})}
     {\Gamma(1-k-\frac{z}{2})\Gamma(k+\frac12-\frac{z}{2})\Gamma(k-\frac{z}{2})}.
\]

The right-hand side expands to
\[
m^{-z} \pi^{\frac32-3z} \zeta(2z) \,
\frac{\Gamma(k-\frac12+\frac{z}{2})\Gamma(\frac{z}{2})\Gamma(-k+\frac12+\frac{z}{2})}
     {\Gamma(-k+1-\frac{z}{2})\Gamma(\frac{1-z}{2})\Gamma(k-\frac{z}{2})}.
\]

Equality of the two expressions reduces to the identity
\[
\Gamma\left(\frac{1-z}{2}\right)\Gamma\left(\frac{1+z}{2}\right)
= (-1)^k \Gamma\left(k+\frac{1-z}{2}\right)\Gamma\left(-k+\frac{1+z}{2}\right),
\]
which follows readily from the functional equation $\Gamma(s+1)=s\Gamma(s)$. Hence \eqref{equ:J-2m-m^2-fe} holds.
\end{proof}

\begin{lemma}\label{lem:J-n-ell-fe}
For $n^2\neq4\ell$ and all $z$ where both sides are meromorphic,
\[
\mathcal{J}_{n,\ell}(z)=\mathcal{G}_{\mu}(z)\Xi_{n^2-4\ell}(z)^{-1}\mathcal{J}_{n,\ell}(1-z).
\]
\end{lemma}

\begin{proof}
Assume $D=n^2-4\ell>0$; the case $D<0$ is similar. By Euler's transformation,
\[
\mathcal{J}_{n,\ell}(z)=\frac{D^{\frac12-z}}{|n|^{1-2z}}\cdot\frac{C_+(z;n,\ell,2k-1)}{C_+(1-z;n,\ell,2k-1)}\,\mathcal{J}_{n,\ell}(1-z).
\]
Hence it suffices to show
\[
\mathcal{G}_{\mu}(z)\Xi_D(z)^{-1}=\frac{D^{\frac12-z}}{|n|^{1-2z}}\cdot\frac{C_+(z;n,\ell,2k-1)}{C_+(1-z;n,\ell,2k-1)}.
\]
Substitute the explicit formulas for $\mathcal{G}_\mu$, $\Xi_D$, and $C_+$ and use the three formulas of $\Gamma(s)$ properly. One can verify that both sides equal
\[
\pi^{1-2z}D^{\frac12-z}\frac{\Gamma(k-\frac12+\frac{z}{2})\Gamma(-k+\frac12+\frac{z}{2})}{\Gamma(1-k-\frac{z}{2})\Gamma(k-\frac{z}{2})}.
\]
Thus the identity holds.
\end{proof}

The following lemma shows that a certain series involving quadratic $L$-functions and $\mathcal{J}_{n,\ell}$ has polynomial growth in vertical strips.

\begin{lemma}\label{lem:holomorphic-nonsingular-series}
Assume $k\geqslant 2$. For fixed $\ell\geqslant1$, define
\[
F_{\ell}(z) = \sum_{n^2\neq4\ell} \mathscr{L}_{n^2-4\ell}(1-z)\,\mathcal{J}_{n,\ell}(z).
\]
Then $F_{\ell}(z)$ is holomorphic in the strip $-\frac12<\operatorname{Re}(z)<\frac32$, and there exists a constant $A>0$ such that
\[
F_{\ell}(z) \ll_{\ell,k} (1+|\operatorname{Im}(z)|)^A \qquad (|\operatorname{Im}(z)|\to\infty)
\]
uniformly in that strip.
\end{lemma}

\begin{proof}
We employ two estimates. First, by the convexity bound \eqref{equ:convexbound}, there exists $A_1>0$ (independent of $n$), such that for any $\delta>0$,
\[
\mathscr{L}_{n^2-4\ell}(1-z) \ll_\delta (1+|\operatorname{Im}(z)|)^{A_1} |n^2-4\ell|^{\max(0,\frac{\operatorname{Re}(z)}{2})+\delta},
\]
away from the possible simple pole at $z=0$, which only exists when $n^2-4\ell$ is a square. For fixed $\ell$, there are only finitely many $n$ such that $n^2-4\ell$ is a square.

Second, from Lemma \ref{lem:J-nu-n-ell}, \eqref{equ:bound-hypergeo} and Stirling's formula, we see that there exists $A_2>0$ (independent of $n$) such that
\[
\mathcal{J}_{n,\ell}(z) \ll_{\ell,k} (1+|\operatorname{Im}(z)|)^{A_2} |n|^{1-2k-\operatorname{Re}(z)}
\]
holds uniformly for every $n$ such that $n^2\neq 4\ell$.

Combining these, for $z$ in any compact subset of $(-\frac12,\frac32)$ avoiding $z=0$, we have
\[
\mathscr{L}_{n^2-4\ell}(1-z)\mathcal{J}_{n,\ell}(z) \ll_{\ell,k,\delta} (1+|\operatorname{Im}(z)|)^A |n|^{1-2k-\operatorname{Re}(z)+2\max(0,\frac{\operatorname{Re}(z)}{2})+\delta},
\]
where $A=A_1+A_2$ is independent of $n$. Since $k\ge2$ and $-\frac12<\operatorname{Re}(z)<\frac32$, the exponent of $|n|$ is $<-1$ for sufficiently small $\delta>0$. Hence the series converges absolutely and defines a holomorphic function there, with the claimed polynomial growth.

For the finitely many square discriminants $n^2-4\ell$, the potential simple pole of $\mathscr{L}_{n^2-4\ell}(1-z)$ at $z=0$ is cancelled by the zero of $1/\Gamma(1-k-z/2)$ appearing in $\mathcal{J}_{n,\ell}(z)$. Thus $F_{\ell}(z)$ is holomorphic throughout the strip.
\end{proof}

\subsection{The Maass form case}

\subsubsection{The archimedean transforms \texorpdfstring{$\mathcal{T}_{\mu_t}$}{T\_{mu\_t}}}

For the Maass case, the spectral parameter is $t\in\mathbb{R}\cup i(0,\tfrac12)$ (or $t_u$ for a discrete Maass form). Set $\mu_t=(2it,0,-2it)$.

\begin{definition}\label{def:T-mu-t}
For $g\in\mathcal{S}$, define $\mathcal{T}_{\mu_t}g$ by
\[
\widetilde{\mathcal{T}_{\mu_t}g}(s)=\widetilde{g}(1-s)\mathcal{G}_{\mu_t}(s),\qquad
\mathcal{G}_{\mu_t}(s)=G_0(s+2it)G_0(s)G_0(s-2it),
\]
or equivalently,
\[
(\mathcal{T}_{\mu_t}g)(y)=\frac1{2\pi i}\int_{(\sigma)} \widetilde{g}(1-s)\mathcal{G}_{\mu_t}(s)y^{-s}ds.
\]
The contour $\operatorname{Re}(s)=\sigma$ is chosen as follows:
\begin{itemize}
   \item If $t\in\mathbb{R}$, any $\sigma>0$ works (the integrand is holomorphic for $\operatorname{Re}(s)>0$).
   \item If $t=ir$ with $0<r<\tfrac12$, take $\sigma>2r$ to stay to the right of the poles at $s=0$ and $s=2r$.
\end{itemize}
\end{definition}

The function $\mathcal{G}_{\mu_t}(s)$ is meromorphic, and it is holomorphic for $\operatorname{Re}(s)>0$. It satisfies $\mathcal{G}_{\mu_t}(s)\mathcal{G}_{\mu_t}(1-s)=1$ and grows polynomially in $|\operatorname{Im}(s)|$ on vertical strips.

\begin{lemma}\label{lem:T_mu_t_props}
Let $t\in\mathbb{R}\cup i(0,\tfrac12)$ and $g\in\mathcal{S}$. Then $\mathcal{T}_{\mu_t}g$ is smooth on $(0,\infty)$ and satisfies:
\begin{itemize}
   \item \textbf{Rapid decay at infinity:} For every $A>0$ and $m\ge0$,
         \[
         (\mathcal{T}_{\mu_t}g)^{(m)}(y)\ll y^{-A}\qquad (y\to\infty).
         \]
   \item \textbf{Good behaviour near zero:} Let $\delta=|\operatorname{Im}(t)|\ge0$. Then as $y\to0^+$,
         \[
         \mathcal{T}_{\mu_t}g(y)=O(y^{-2\delta}),\quad
         (\mathcal{T}_{\mu_t}g)'(y)=O(y^{-1-2\delta}),\quad
         (\mathcal{T}_{\mu_t}g)''(y)=O(y^{-2-2\delta}).
         \]
\end{itemize}
\end{lemma}

\begin{proof}
Recall that $\widetilde{\mathcal{T}_{\mu_t}g}(s)=\widetilde{g}(1-s)\mathcal{G}_{\mu_t}(s)$. Since $g\in\mathcal{S}$, $\widetilde{g}$ decays rapidly on vertical strips, while $\mathcal{G}_{\mu_t}$ has polynomial growth.

\emph{Rapid Decay at infinity.} Shift the contour to $\operatorname{Re}(s)=A+m+1$. Then
\[
(\mathcal{T}_{\mu_t}g)^{(m)}(y)=\frac{(-1)^m}{2\pi i}\int_{(A+m+1)} s\cdots(s+m-1)\widetilde{g}(1-s)\mathcal{G}_{\mu_t}(s)y^{-s-m}ds,
\]
and $|y^{-s-m}|=y^{-A-2m-1}$. The integrand decays rapidly, giving $(\mathcal{T}_{\mu_t}g)^{(m)}(y)\ll y^{-A}$.

\emph{Good behaviour near zero.} Shift the contour to $\operatorname{Re}(s)=\sigma_0\in(-3,-2-2\delta)$. The function $\mathcal{G}_{\mu_t}(s)$ has simple poles at $s=0,-2,-4,\dots$ and $s=\pm2it-2n$ ($n\ge0$), all with real part $\le2\delta$. Collecting residues at poles with $\operatorname{Re}(w)>-\frac52$ yields the stated bounds, while the remaining contour integrals are of lower order.
\end{proof}

These transforms are the archimedean building blocks for the Voronoi summation formulas in Theorems \ref{thm:holo-voronoi} and \ref{thm:maass-voronoi}.

\subsubsection{The Maass kernel \texorpdfstring{$\mathcal{J}_{n,\ell,t}(z)$}{J\_\{n,ℓ,t\}(z)}}
For the Maass case, we replace the Bessel function $J_{2k-1}(x)$ by the kernel
\[
\mathbb{J}_t(x)=\frac{i}{\pi}\frac{t}{\cosh(\pi t)}\bigl(J_{2it}(x)-J_{-2it}(x)\bigr),
\]
where $t\in\mathbb{R}\cup i{0,\frac12}$.

Define
\begin{equation}\label{equ:def-J-n-ell-t}
\mathcal{J}_{n,\ell,t}(z)=2\int_{0}^{\infty}y^{z-1}\mathbb{J}_t(4\pi\sqrt{\ell}y)\cos(2\pi ny)dy.
\end{equation}
Since $\mathbb{J}_t$ is a linear combination of $J_{\pm2it}$, we obtain
\[
    \mathcal{J}_{n,\ell,t}(z)=\frac{i}{\pi}\frac{t}{\cosh(\pi t)}\big(\mathcal{J}_{n,\ell}^{2it}(z)-\mathcal{J}_{n,\ell}^{-2it}(z)\big).
\]
By the explicit computation of $\mathcal{J}_{n,\ell}^{\nu}(z)$, we have
\begin{lemma}
\begin{itemize}
\item If $n^2\neq 4\ell$, the function $\mathcal{J}_{n,\ell,t}(z)$ is meromorphic when $-\frac12<\mathrm{Re}(z)<\frac32$ with simple poles at $z=\pm 2it$ and admits the functional equation
\[
\mathcal{J}_{n,\ell,t}(z)=\mathcal{G}_{\mu_t}(z)\Xi_{n^2-4\ell}(z)^{-1}\mathcal{J}_{n,\ell,t}(1-z),
\] 
\item If $\ell=m^2$ is a perfect square, the function $\mathcal{J}_{\pm 2m,m^2,t}(z)$ is meromorphic when $-\frac12<\mathrm{Re}(z)<\frac32$ with simple poles at $z=\pm 2it$ and $z=\frac{1}{2}$.
\end{itemize} 
\end{lemma}

\begin{proof}
The proof follows Lemma \ref{lem:J-nu-n-ell}: apply Lemma \ref{lem:WS} to each Bessel order, use Euler's transformation, and combine gamma factors.
\end{proof}

\begin{lemma}\label{lem:J-2m-m^2-t-fe}
For $m\ge1$ and $t\in\mathbb{R}$,
\[
\mathcal{J}_{\pm2m,m^2,t}(z)=\frac{t\tanh(\pi t)}{\pi^2}\,m^{-z}\,\mathcal{G}_{\mu_t}(z)\,G_0(2z)^{-1}.
\]
\end{lemma}

\begin{proof}
By Lemma \ref{lem:J-nu-n-ell}, 
\[
\mathcal{J}_{\pm2m,m^2}^{\nu}(z)-\mathcal{J}_{\pm2m,m^2}^{-nu}(z)=2^{-z}\pi^{\frac{1}{2}-z}m^{-z}
        \frac{\Gamma\bigl(\frac12-z\bigr)}
             {\Gamma\bigl(\frac{-\nu-z+1}{2}\bigr)\Gamma\bigl(\frac{\nu-z+1}{2}\bigr)} \left[\frac{\Gamma\bigl(\frac{\nu+z}{2}\bigr)}
             {\Gamma\bigl(\frac{\nu-z}{2}+1\bigr)}-\frac{\Gamma\bigl(\frac{-\nu+z}{2}\bigr)}
             {\Gamma\bigl(\frac{-\nu-z}{2}+1\bigr)}\right].
\]
Substitute $\nu=2it$ and use $\Gamma(s)\Gamma(1-s)=\pi/\sin(\pi s)$, we have
\[
\frac{\Gamma(it+\frac{z}{2})}{\Gamma(it-\frac{z}{2}+1)}-\frac{\Gamma(-it+\frac{z}{2})}{\Gamma(-it-\frac{z}{2}+1)}=-2i\sinh(\pi t)\frac{\Gamma(it+\frac{z}{2})\Gamma(it-\frac{z}{2})}{\Gamma(\frac{1-z}{2})\Gamma(\frac{1+z}{2})}.
\]
Plugging this into $\mathcal{J}_{2m,m^2,t}(z)$ and using the duplication formula for $\Gamma(z)$ yields
\[
G_0(2z)\mathcal{J}_{2m,m^2,t}(z)=t\tanh(\pi t)\cdot\pi^{-\frac12-3z}m^{-z}\frac{\Gamma(\frac{z}{2})\Gamma(it+\frac{z}{2})\Gamma(it-\frac{z}{2})}{\Gamma(-it-\frac{z}{2}+\frac12)\Gamma(it-\frac{z}{2}+\frac12)\Gamma(\frac{1-z}{2})}.
\]
Comparing with $\mathcal{G}_{\mu_t}(z)=\pi^{\frac32-3z}\frac{\Gamma(\frac{z}{2})\Gamma(it+\frac{z}{2})\Gamma(it-\frac{z}{2})}{\Gamma(-it-\frac{z}{2}+\frac12)\Gamma(it-\frac{z}{2}+\frac12)\Gamma(\frac{1-z}{2})}$ gives the desired identity.
\end{proof}

\begin{lemma}\label{lem:J-n-ell-t-0}
Let $d^2=n^2-4\ell$ with $d\ge0$, and let $(a,b)$ be the unique positive integers with $|n|=a+b$, $d=b-a$, $ab=\ell$, $a\le b$. Then
\[
\mathcal{J}_{n,\ell,t}(0)=\frac{1}{\pi}\left[\left(\frac{a}{b}\right)^{it}+\left(\frac{b}{a}\right)^{it}\right].
\]
\end{lemma}

\begin{proof}
If $d=0$, then $a=b$, $|n|=2a$, $\ell=a^2$. Direct evaluation using Lemma \ref{lem:J-nu-n-ell} gives $\mathcal{J}_{\pm2a,a^2,t}(0)=2/\pi$, which matches the formula since $(a/b)^{it}=1$.

If $d>0$, from Lemma \ref{lem:J-nu-n-ell} we have
\[
\mathcal{J}_{\pm n,\ell,t}(0)=\frac{it}{\pi\cosh(\pi t)}\Bigl(C_{+}(0;n,\ell,2it)\,{}_2F_1(\cdots)-C_{+}(0;n,\ell,-2it)\,{}_2F_1(\cdots)\Bigr),
\]
where $C_{+}(0;n,\ell,2it)=\ell^{it}|n|^{-2it}\frac{\cosh(\pi t)}{it}$. Using \eqref{equ:hyper-geo-square-root},
\[
{}_2F_1\left(it,it+\tfrac12;1+2it;\tfrac{4\ell}{n^2}\right)=\left(\tfrac{2}{1+\frac{b-a}{a+b}}\right)^{2it}=\left(\tfrac{a+b}{b}\right)^{2it}.
\]
Substituting yields the stated result.
\end{proof}

\subsubsection{The average kernel \texorpdfstring{$\mathcal{J}_{n,\ell,h}(z)$}{J\_\{n,ℓ,h\}(z)}}

Let $h(t)$ be a Kuznetsov-admissible test function with $h(\pm i/2)=0$. Define
\begin{equation}\label{equ:J-n-ell-h-defn}
\mathcal{J}_{n,\ell,h}(z)=\int_{-\infty}^{\infty}h(t)\,\mathcal{J}_{n,\ell,t}(z)\,dt.
\end{equation}
By Lemma \ref{lem:J-nu-n-ell}, this integral is meromorphic for $0<\operatorname{Re}(z)<\frac32$, with a possible simple pole at $z=\frac12$ when $\ell$ is a perfect square and $n^2=4\ell$.

Exchanging the order of integration gives, for $0<\operatorname{Re}(z)<\frac12$,
\[
\mathcal{J}_{n,\ell,h}(z)=2\int_0^{\infty}y^{z-1}\mathcal{J}_h(4\pi\sqrt{\ell}y)\cos(2\pi ny)dy,
\]
where $\mathcal{J}_h(y)=\int_{-\infty}^{\infty}h(t)\mathbb{J}_t(y)dt$. The absolute convergence justifies the exchange in this region.

\begin{lemma}\label{lem:Jh-estimate}
Let $h$ be holomorphic for $|\operatorname{Im}(t)|<1+\varepsilon_0$ with $h(\pm i/2)=0$. Then for every $\alpha\in(2,2+2\varepsilon_0)$,
\[
\mathcal{J}_h^{(j)}(x)\ll x^{\alpha-j}\qquad (x\to0^+),\quad j=0,1,2.
\]
\end{lemma}

\begin{proof}
Set $H(t)=th(t)/\cosh(\pi t)$. The condition $h(\pm i/2)=0$ removes the poles of $H(t)$ in the strip. For $0<x\le1$, shift the contour in the $J_{2it}$-integral to $\operatorname{Im}(t)=-\alpha/2$ and the $J_{-2it}$-integral to $\operatorname{Im}(t)=\alpha/2$:
\[
\mathcal{J}_h(x)=\frac{i}{\pi}\left(\int_{-\infty}^{\infty}H(t-i\alpha/2)J_{\alpha+2it}(x)dt-\int_{-\infty}^{\infty}H(t+i\alpha/2)J_{\alpha-2it}(x)dt\right).
\]
The new Bessel orders have real part $\alpha$, so $J_{\alpha\pm2it}(x)\ll x^\alpha(1+|t|)^Be^{\pi|t|}$ for some $B>0$. Since $h$ is of rapid decays, we have $H(t\pm i\alpha/2)\ll_{h,N}(1+|t|)^{-B-2}e^{-\pi|t|}$. So $\mathcal{J}_h(x)\ll x^\alpha$. The derivative bounds can be obtained by differentiating.
\end{proof}

From Lemma \ref{lem:Jh-estimate} and the integral representation, $\mathcal{J}_{n,\ell,h}(z)$ extends meromorphically to $-2-2\varepsilon<\operatorname{Re}(z)<\frac32$, with a possible simple pole at $z=\frac12$ when $\ell$ is a perfect square and  $n^2=4\ell$.

\begin{lemma}\label{lem:Jnlellh-growth}
Assume $h$ is even, holomorphic and rapidly decaying for $|\operatorname{Im}(t)|<1+\varepsilon$, with $h(\pm i/2)=0$. For fixed $\ell\ge1$, define
\[
F_{\ell,h}(z)=\sum_{n^2\neq4\ell}\mathscr{L}_{n^2-4\ell}(1-z)\,\mathcal{J}_{n,\ell,h}(z).
\]
Then $F_{\ell,h}(z)$ is meromorphic in $-\frac12<\operatorname{Re}(z)<\frac32$, with at most a simple pole at $z=0$, and has polynomial growth in $|\operatorname{Im}(z)|$ away from $z=0$.
\end{lemma}

\begin{proof}
Similar to Lemma \ref{lem:holomorphic-nonsingular-series}.
\end{proof}

The averaged kernel satisfies a functional equation.

\begin{lemma}\label{lem:Jnlellh_fe}
For $n^2\neq4\ell$ and $-\frac12<\operatorname{Re}(z)<0$,
\[
\mathcal{J}_{n,\ell,h}(1-z)=\Xi_{n^2-4\ell}(1-z)^{-1}\,\mathcal{J}_{n,\ell,h_{1-z}}(z),
\]
where $h_{1-z}(t)=h(t)\,\mathcal{G}_{\mu_t}(1-z)$.
\end{lemma}

\begin{proof}
Insert the functional equation for $\mathcal{J}_{n,\ell,t}(z)$:
\[
\mathcal{J}_{n,\ell,h}(1-z)=\int_{-\infty}^{\infty}h(t)\,\mathcal{G}_{\mu_t}(1-z)\,\Xi_D(1-z)^{-1}\mathcal{J}_{n,\ell,t}(z)dt=\Xi_D(1-z)^{-1}\mathcal{J}_{n,\ell,h_{1-z}}(z),
\]
where $D=n^2-4\ell$. The interchange is justified by absolute convergence and rapid decay.
\end{proof}

The transformed test function $h_{1-z}$ may not be Kuznetsov-admissible due to poles of $\mathcal{G}_{\mu_t}(1-z)$ at $t=\pm i(1-z)/2$. Nevertheless, for $\operatorname{Re}(z)<0$, we have the decomposition
\begin{equation}\label{equ:decom-Jh-1-z}
\mathcal{J}_{h_{1-z}}(y)=\mathcal{Z}_{\ell,h}(z)J_{1-z}(y)+R_{\ell,h}(y),
\end{equation}
where $J_{1-z}$ is the Bessel function of order $1-z$, $R_{\ell,h}^{(j)}(y)\ll y^{\alpha-j}$ as $y\to0^+$ for any $\alpha<2+2\varepsilon$, and
\begin{equation}\label{equ:Z-ell-h}
\mathcal{Z}_{\ell,h}(z)=-2(1-z)h\left(\frac{1-z}{2i}\right)\frac{G_0(1-z)G_0(2-2z)}{\cosh\left(\frac{\pi(1-z)}{2i}\right)}.
\end{equation}

\section{Averaged Voronoi Identities}
\subsection{Holomorphic case (Theorem \ref{thm:holo-voronoi})}

We prove \eqref{equ:holo-voronoi} in six steps.

\emph{Step 1: Petersson trace formula.}\quad Recall that
\[
I_\ell(g)=\sumh_{f\in B_{2k}} a_f(\ell)\sum_{n\ge1} A_f(n)g(n).
\]
Writing $A_f(n)=\sum_{d^2\mid n}a_f((n/d^2)^2)$ and $n=d^2m$, we obtain
\[
I_\ell(g)=\sum_{d,m\ge1} g(d^2m)\sumh_{f\in B_{2k}} a_f(\ell)a_f(m^2).
\]
Applying the Petersson trace formula to $(\ell,m^2)$ gives
\[
\sumh_{f\in B_{2k}} a_f(\ell)a_f(m^2)=\delta_{\ell,m^2}+2\pi i^{-2k}\sum_{c\ge1}\frac{S(\ell,m^2;c)}{c}J_{2k-1}\left(\frac{4\pi m\sqrt\ell}{c}\right).
\]
Hence
\[
I_\ell(g)=\Delta_\ell(g)+\mathcal{G}_\ell(g),\qquad
\Delta_\ell(g):=\mathbf1_{\ell=\square}\sum_{d\ge1}g(d^2\sqrt\ell),
\]
\[
\mathcal{G}_\ell(g):=2\pi i^{-2k}\sum_{d,m\ge1}g(d^2m)\sum_{c\ge1}\frac{S(\ell,m^2;c)}{c}J_{2k-1}\left(\frac{4\pi m\sqrt\ell}{c}\right).
\]
Opening the Kloosterman sum yields
\[
\mathcal{G}_\ell(g)=\pi i^{-2k}\sum_{c\ge1}\frac1c\sideset{}{^*}\sum_{x\bmod c}e\left(\frac{\ell\bar x}{c}\right)\sum_{d\ge1}\sideset{}{'}\sum_{m\in\mathbb{Z}}g(d^2|m|)J_{2k-1}\left(\frac{4\pi|m|\sqrt\ell}{c}\right)e\left(\frac{xm^2}{c}\right).
\]

\vspace{0.5cm}
\emph{Step 2: Poisson summation in the quadratic variable.}\quad Since $g\in\mathcal{S}$, Poisson summation gives
\[
\mathcal{G}_\ell(g)=\pi i^{-2k}\sum_{c\ge1}\frac1c\sideset{}{^*}\sum_{x\bmod c}e\left(\frac{\ell\bar x}{c}\right)\sum_{d\ge1}\sum_{a\bmod c}e\left(\frac{xa^2+an}{c}\right)\sum_{n\in\mathbb{Z}}\widehat{W}_{c,d,\ell}\left(\frac{n}{c}\right),
\]
where
\[
\widehat{W}_{c,d,\ell}(\xi)=2\int_0^\infty g(d^2t)J_{2k-1}\left(\frac{4\pi t\sqrt\ell}{c}\right)\cos(2\pi t\xi)dt.
\]
Using the arithmetic identity \eqref{equ:Ramanujansum},
\[
\Sigma_c(\ell,n):=\sideset{}{^*}\sum_{x\bmod c}e\left(\frac{\ell\bar x}{c}\right)\sum_{a\bmod c}e\left(\frac{xa^2+an}{c}\right)=c\sum_{r\mid c}\mu\left(\frac{c}{r}\right)\rho_r(n,\ell),
\]
with $\rho_r(n,\ell)=\#\{u\bmod r:u^2+nu+\ell\equiv0\pmod r\}$. Substituting and setting $c=rq$ leads to
\[
\mathcal{G}_\ell(g)=\pi i^{-2k}\sum_{r\ge1}\frac1r\sum_{n\in\mathbb{Z}}\rho_r(n,\ell)\sum_{q\ge1}\frac{\mu(q)}{q}\sum_{d\ge1}\widehat{W}_{rq,d,\ell}\left(\frac{n}{rq}\right).
\]

By the Mellin inversion for $g$, we have
\begin{equation*}
\begin{split}
&\sum_{n\in\mathbb{Z}}\rho_r(n,\ell)\sum_{q\ge1}\frac{\mu(q)}{q}\sum_{d\ge1}\widehat{W}_{rq,d,\ell}\left(\frac{n}{rq}\right)\\
=&2\sum_{q\geqslant 1}\frac{\mu(q)}{q}\sum_{d\geqslant 1}\frac1{d^2}\int_0^\infty
\left(\frac1{2\pi i}\int_{(\sigma)}
\widetilde g(1-z)y^{z-1}dz\right)
J_{2k-1}\left(\frac{4\pi y\sqrt{\ell}}{rqd^2}\right)
\cos\left(\frac{2\pi ny}{rqd^2}\right)dy\\
=&\frac1{2\pi i}\int_{(\sigma)}
\widetilde g(1-z)r^{z}\mathcal J_{n,\ell}(z)\sum_{q\geqslant 1}\frac{\mu(q)}{q^{1-z}}\sum_{d\geqslant 1}\frac1{d^{2-2z}}dz\\
=&\frac1{2\pi i}\int_{(\sigma)}
\widetilde g(1-z)
r^z\frac{\zeta(2-2z)}{\zeta(1-z)}
\mathcal J_{n,\ell}(z)dz,
\end{split}
\end{equation*}
for $\sigma\in(-1/2,0)$, where
\begin{equation}
\label{eqn: def-Jnl}
\mathcal J_{n,\ell}(z)
=2\int_0^\infty y^{z-1}J_{2k-1}(4\pi\sqrt{\ell}y)\cos(2\pi ny)dy.
\end{equation}

Substituting this formula back into $\mathcal{G}_{\ell}(g)$ and summing over $r$
first, we obtain
\begin{equation*}
\begin{split}
\mathcal G_\ell(g)=&\frac{\pi i^{-2k}}{2\pi i}\int_{(\sigma)}
\widetilde g(1-z)\sum_{r\geqslant1}\frac{1}{r^{1-z}}
\sum_{n\in\mathbb{Z}}\rho_r(n,\ell)\frac{\zeta(2-2z)}{\zeta(1-z)}
\mathcal J_{n,\ell}(z)dz\\
=&\frac{\pi i^{-2k}}{2\pi i}\int_{(\sigma)}
\widetilde g(1-z)
\sum_{n\in\mathbb{Z}}
\mathscr L_{n^2-4\ell}(1-z)\mathcal J_{n,\ell}(z)dz.
\end{split}
\end{equation*}
Here, we used the definition 
\[
    \mathscr{L}_{n^2-4\ell}(s)=\frac{\zeta(2s)}{\zeta(s)}\sum_{r=1}^{\infty}\frac{\rho_{r}(n,\ell)}{r^s},\quad\mathrm{Re}(z)>0.
\]
\vspace{0.5cm}

\emph{Step 3: Apply the same to $I_\ell(\mathcal{T}_\mu g)$.}\quad
By Lemma \ref{lem:Tmug-props}, repeating the same steps gives
\[
I_\ell(\mathcal{T}_\mu g)=\Delta_\ell(\mathcal{T}_\mu g)+\mathcal{G}_\ell(\mathcal{T}_\mu g),
\]
with
\[
\mathcal{G}_\ell(\mathcal{T}_\mu g)=\frac{\pi i^{-2k}}{2\pi i}\int_{(\sigma)}\widetilde{\mathcal{T}_\mu g}(1-z)\sum_{n\in\mathbb{Z}}\mathscr{L}_{n^2-4\ell}(1-z)\mathcal{J}_{n,\ell}(z)dz.
\]

\vspace{0.5cm}

\emph{Step 4: Archimedean reciprocity.}\quad
Let $\mathcal{G}_\ell^{\mathrm{ns}}(g)$ denote the contribution of $n^2\neq4\ell$ in $\mathcal{G}_\ell(g)$. By Lemma \ref{lem:J-n-ell-fe} and \eqref{equ:LD_fe},
\[
\mathscr{L}_{n^2-4\ell}(1-z)\mathcal{J}_{n,\ell}(z)=\mathcal{G}_\mu(z)\mathscr{L}_{n^2-4\ell}(z)\mathcal{J}_{n,\ell}(1-z)\qquad(n^2\neq4\ell).
\]
Combined with $\widetilde{\mathcal{T}_\mu g}(s)=\widetilde g(1-s)\mathcal{G}_\mu(s)$, we obtain
\[
\mathcal{G}_\ell^{\mathrm{ns}}(g)=\frac{\pi i^{-2k}}{2\pi i}\int_{(\sigma)}\widetilde{\mathcal{T}_\mu g}(z)\sum_{n^2\neq4\ell}\mathscr{L}_{n^2-4\ell}(z)\mathcal{J}_{n,\ell}(1-z)dz.
\]
Shifting the contour (justified by Lemma \ref{lem:holomorphic-nonsingular-series} and \ref{lem:Tmug-props}) yields
\[
\mathcal{G}_\ell^{\mathrm{ns}}(g)=\mathcal{G}_\ell^{\mathrm{ns}}(\mathcal{T}_\mu g).
\]
If $\ell$ is not a perfect square, this already gives $I_\ell(g)=I_\ell(\mathcal{T}_\mu g)$.

\vspace{0.5cm}

\emph{Step 5: Singular and diagonal corrections.}\quad
Now assume $\ell=m^2$. Then
\[
I_\ell(g)-I_\ell(\mathcal{T}_\mu g)=\frac{\pi i^{-2k}}{2\pi i}\int_{(\sigma)}\left(\widetilde g(1-z)-\widetilde{\mathcal{T}_\mu g}(1-z)\right)\sum_{n=\pm2m}\mathscr{L}_0(1-z)\mathcal{J}_{n,m^2}(z)dz+\Delta_{m^2}(g)-\Delta_{m^2}(\mathcal{T}_\mu g).
\]
Since $\mathscr{L}_0(s)=\zeta(2s-1)$ and $\mathcal{J}_{-2m,m^2}=\mathcal{J}_{2m,m^2}$, define
\[
\mathcal{S}_m(g)=\frac1{2\pi i}\int_{(\sigma)}\widetilde g(1-z)\mathcal{R}_m(z)dz,\qquad
\mathcal{R}_m(z):=2\pi i^{-2k}\zeta(1-2z)\mathcal{J}_{2m,m^2}(z).
\]
By Lemma \ref{lem:J-2m-m^2-fe}, $\mathcal{R}_m(z)=\mathcal{G}_\mu(z)m^{-z}\zeta(2z)$. A direct computation using $\mathcal{G}_\mu(z)\mathcal{G}_\mu(1-z)=1$ shows that the combined kernel satisfies
\[
\mathcal{R}_m(z)+m^{z-1}\zeta(2-2z)=\mathcal{G}_\mu(z)\bigl(\mathcal{R}_m(1-z)+m^{-z}\zeta(2z)\bigr).
\]
Notice that the term $m^{z-1}\zeta(2-2z)$ comes from the 
Noting that $\Delta_{m^2}(g)=\frac1{2\pi i}\int_{(\sigma)}\widetilde g(1-z)m^{z-1}\zeta(2-2z)dz$, and using $\widetilde{\mathcal{T}_\mu g}(z)=\widetilde g(1-z)\mathcal{G}_\mu(z)$, we obtain after contour shifting
\[
\mathcal{S}_m(\mathcal{T}_\mu g)+\Delta_{m^2}(\mathcal{T}_\mu g)=\mathcal{S}_m(g)+\Delta_{m^2}(g).
\]
Hence $I_{m^2}(g)=I_{m^2}(\mathcal{T}_\mu g)$.

This completes the proof of Theorem \ref{thm:holo-voronoi}.

\subsection{Maass case (Theorem \ref{thm:maass-voronoi})}

The proof follows the same argument as the holomorphic case with the following modifications:

\begin{enumerate}
    \item \emph{Trace formula:}\quad Replace the Petersson trace formula with the Kuznetsov trace formula, introducing both discrete (Maass cusp forms) and continuous (Eisenstein series) spectra.

    \item \emph{Archimedean transform:}\quad Replace $T_\mu$ with $T_{\mu_t}$, depending on the spectral parameter $t$, where $\mu_t = (2it,0,-2it)$ and the associated kernel $G_{\mu_t}(s)$ replaces $G_\mu(s)$. Since $I^{\vee}_{\ell}(g;h)\neq I_{\ell}(\mathcal{T}_{\mu_t}g;h)$ for some fixed $t$, we need to treat $I^{\vee}_{\ell}(g;h)$ differently.

      \item \emph{Continuous spectrum contribution:}\quad Unlike the holomorphic case, the pole at $z=0$ coming from the square-discriminant terms $\mathscr{L}_{n,\ell}(1-z)$ is not cancelled by the Maass kernel $J_{n,4\ell,t}(z)$. Consequently, the contour shift crosses this pole and produces an additional residue term. After averaging over the continuous spectrum, these residues give rise to the explicit Eisenstein contribution $P_t(g)$ appearing in the Voronoi identity.
\end{enumerate}

We proceed in six steps.

\vspace{0.5cm}
\emph{Step 1: Kuznetsov trace formula.}\quad
Define $I_\ell(g;h)$ as the left-hand side of \eqref{equ:maass-voronoi}. Using
\[
A_u(n)=\sum_{d^2\mid n}\lambda_u\left(\left(\frac{n}{d^2}\right)^2\right),\qquad
A_t(n)=\sum_{d^2\mid n}\tau_t\left(\left(\frac{n}{d^2}\right)^2\right),
\]
and writing $n=d^2m$, the Kuznetsov formula \eqref{equ:level-one-kuznetsov} gives
\[
I_\ell(g;h)=H_0(h)\Delta_\ell(g)+\mathcal{G}_{\ell,h}(g),
\]
where
\[
H_0(h):=\frac1{\pi^2}\int_{-\infty}^{\infty}t\tanh(\pi t)h(t)dt,\quad
\Delta_\ell(g):=\mathbf1_{\ell=\square}\sum_{d\ge1}g(d^2\sqrt\ell),
\]
and
\[
\mathcal{G}_{\ell,h}(g)=\sum_{c\ge1}\frac1c\sideset{}{^*}\sum_{x\bmod c}e\left(\frac{\ell\bar x}{c}\right)\sum_{d\ge1}\sideset{}{'}\sum_{m\in\mathbb{Z}}g(d^2|m|)\mathcal{J}_h\left(\frac{4\pi|m|\sqrt\ell}{c}\right)e\left(\frac{xm^2}{c}\right).
\]

\vspace{0.5cm}
\emph{Step 2: Poisson summation in the quadratic variable.}\quad
Repeating the Poisson summation argument from the holomorphic case (which is valid because $g\in\mathcal{S}$) yields
\[
\mathcal{G}_{\ell,h}(g)=\sum_{r\ge1}\frac1r\sum_{n\in\mathbb{Z}}\rho_r(n,\ell)\sum_{q\ge1}\frac{\mu(q)}{q}\sum_{d\ge1}\frac1{d^2}\int_0^{\infty}g(y)\mathcal{J}_h\left(\frac{4\pi y\sqrt\ell}{rqd^2}\right)\cos\left(\frac{2\pi ny}{rqd^2}\right)dy;
\]
and furthermore,
\begin{equation}\label{equ:maass-G-integral}
\mathcal G_{\ell,h}(g)
=\frac1{4\pi i}\int_{(\sigma)}
\widetilde g(1-z)
\sum_{n\in\mathbb{Z}}
\mathscr L_{n^2-4\ell}(1-z)\mathcal J_{n,\ell,h}(z)dz,
\end{equation}
where $-\frac12<\sigma<0$.

\vspace{0.5cm}
\emph{Step 3: Apply the same to $I_\ell^\vee(g;h)$.}\quad
Unlike the holomorphic case, the transform $\mathcal{T}_{\mu_t}$ depends on the spectral parameter $t$, so the manipulation for $I_\ell^\vee(g;h)$ is not identical. 

Let $\varepsilon_0>0$ such that $h$ is holomorphic for $|\operatorname{Im}(t)|<1+\varepsilon_0$, and take $\sigma\in(-\varepsilon_0,0)$. By definition,
\[
I_\ell^\vee(g;h)=\frac1{2\pi i}\int_{(\sigma)}\widetilde{g}(z)\mathcal{C}_{\ell,h}(z)dz,
\]
where
\[
\mathcal{C}_{\ell,h}(z):=\sum_{d,m\ge1}d^{2z-2}m^{z-1}\left[\sum_{u\in\mathcal{B}}\omega_u h_{1-z}(t_u)\lambda_u(\ell)\lambda_u(m^2)+\frac1\pi\int_{-\infty}^{\infty}\frac{h_{1-z}(t)\tau_t(\ell)\tau_t(m^2)}{|\zeta(1+2it)|^2}dt\right],
\]
with $h_{1-z}(t)=h(t)\mathcal{G}_{\mu_t}(1-z)$. 

Applying the Kuznetsov trace formula gives
\[
\mathcal{C}_{\ell,h}(z)=\zeta(2-2z)\sum_{m\ge1}m^{z-1}\left[\delta_{\ell,m^2}H_0(h_{1-z})+\sum_{c\ge1}\frac{S(\ell,m^2;c)}{c}\mathcal{J}_{h_{1-z}}\left(\frac{4\pi m\sqrt{\ell}}{c}\right)\right].
\]

Define
\[
\mathcal{D}_{\ell,h}(z):=\sum_{m\ge1}m^{z-1}\sum_{c\ge1}\frac{S(\ell,m^2;c)}{c}\mathcal{J}_{h_{1-z}}\left(\frac{4\pi m\sqrt{\ell}}{c}\right).
\]
Opening the Kloosterman sum yields
\[
\mathcal{D}_{\ell,h}(z)=\frac12\sum_{c\ge1}\frac1c\sideset{}{^*}\sum_{x\bmod c}e\left(\frac{\ell\overline{x}}{c}\right)\sideset{}{'}\sum_{m\in\mathbb{Z}}|m|^{z-1}\mathcal{J}_{h_{1-z}}\left(\frac{4\pi|m|\sqrt{\ell}}{c}\right)e\left(\frac{xm^2}{c}\right).
\]

To apply Poisson summation, we use the decomposition (see \eqref{equ:decom-Jh-1-z})
\[
\mathcal{J}_{h_{1-z}}(y)=\mathcal{Z}_{\ell,h}(z)J_{1-z}(y)+R_{\ell,h}(y),
\]
where $J_{1-z}$ is the Bessel function of order $1-z$, $R_{\ell,h}^{(j)}(y)\ll y^{\alpha-j}$ as $y\to0^+$ for $j=0,1,2$ and any $\alpha<2+2\varepsilon_0$, and
\[
\mathcal{Z}_{\ell,h}(z)=-2(1-z)h\left(\frac{1-z}{2i}\right)\frac{G_0(1-z)G_0(2-2z)}{\cosh\left(\frac{\pi(1-z)}{2i}\right)}.
\]

Applying Poisson summation (and noting the $m=0$ contribution) gives
\begin{equation}\label{equ:poisson-dual}
\begin{split}
  &\sideset{}{'}\sum_{m\in\mathbb{Z}}|m|^{z-1}\mathcal{J}_{h_{1-z}}\left(\frac{4\pi|m|\sqrt{\ell}}{c}\right)e\left(\frac{xm^2}{c}\right)+\lim_{y\to0^+}y^{z-1}\mathcal{J}_{h_{1-z}}\left(\frac{4\pi y\sqrt{\ell}}{c}\right)\\
  =&\frac{2}{c}\sum_{a\bmod c}\sum_{n\in\mathbb{Z}}e\left(\frac{a^2x+an}{c}\right)\int_0^{\infty}y^{z-1}\mathcal{J}_{h_{1-z}}\left(\frac{4\pi y\sqrt{\ell}}{c}\right)\cos(2\pi ny)dy.
\end{split}
\end{equation}

Using the arithmetic identity \eqref{equ:Ramanujansum}, we obtain
\begin{equation*}
\begin{split}
\mathcal{D}_{\ell,h}(z)=&\sum_{r\ge1}\frac1r\sum_{n\in\mathbb{Z}}\rho_r(n,\ell)\sum_{q\ge1}\frac{\mu(q)}{q}\int_0^{\infty}y^{z-1}\mathcal{J}_{h_{1-z}}\left(\frac{4\pi y\sqrt{\ell}}{rq}\right)\cos(2\pi ny)dy\\
&+\frac12\sum_{c\ge1}\sideset{}{^*}\sum_{x\bmod c}e\left(\frac{\ell\overline{x}}{c}\right)\lim_{y\to0^+}y^{z-1}\mathcal{J}_{h_{1-z}}\left(\frac{4\pi y\sqrt{\ell}}{c}\right). 
\end{split}
\end{equation*}

Substituting back into $\mathcal{C}_{\ell,h}(z)$ and simplifying, the middle term becomes
\[
\frac{1}{4\pi i}\int_{(\sigma)}\widetilde{g}(z)\sum_{n\in\mathbb{Z}}\mathscr{L}_{n^2-4\ell}(1-z)\mathcal{J}_{n,\ell,h_{1-z}}(z)dz.
\]

Finally, we decompose $I_\ell^\vee(g;h)$ into four components:
\[
I_\ell^\vee(g;h)=I_\ell^{\vee,\mathrm{ns}}(g;h)+I_\ell^{\vee,\mathrm{sing}}(g;h)+I_\ell^{\vee,0}(g;h)+I_\ell^{\vee,\mathrm{diag}}(g;h),
\]
where
\begin{align*}
I_\ell^{\vee,\mathrm{ns}}(g;h)&=\frac{1}{4\pi i}\int_{(\sigma)}\widetilde{g}(z)\sum_{n^2\neq4\ell}\mathscr{L}_{n^2-4\ell}(1-z)\mathcal{J}_{n,\ell,h_{1-z}}(z)dz,\\
I_\ell^{\vee,\mathrm{sing}}(g;h)&=\frac{1}{4\pi i}\int_{(\sigma)}\widetilde{g}(z)\sum_{n^2=4\ell}\mathscr{L}_{n^2-4\ell}(1-z)\mathcal{J}_{n,\ell,h_{1-z}}(z)dz,\\
I_\ell^{\vee,0}(g;h)&=\frac{1}{4\pi i}\int_{(\sigma)}\widetilde{g}(z)\zeta(2-2z)\sum_{c\ge1}\sideset{}{^*}\sum_{x\bmod c}e\left(\frac{\ell\overline{x}}{c}\right)\lim_{y\to0^+}y^{z-1}\mathcal{J}_{h_{1-z}}\left(\frac{4\pi y\sqrt{\ell}}{c}\right)dz,\\
I_\ell^{\vee,\mathrm{diag}}(g;h)&=\frac{1}{2\pi i}\int_{(\sigma)}\widetilde{g}(z)\zeta(2-2z)\sum_{m\ge1}m^{z-1}\delta_{\ell,m^2}H_0(h_{1-z})dz.
\end{align*}

\vspace{0.5cm}
\emph{Step 4: Archimedean reciprocity.}\quad
Denote
\[
\mathcal{G}_{\ell,h}^{\mathrm{ns}}(g)=\frac1{4\pi i}\int_{(\sigma)}\widetilde{g}(1-z)\sum_{n^2\neq4\ell}\mathscr{L}_{n^2-4\ell}(1-z)\mathcal{J}_{n,\ell,h}(z)dz,
\]
with $\sigma\in(-\frac12,0)$. By contour shifting and Lemma~\ref{lem:Jnlellh-growth},
\[
\mathcal{G}_{\ell,h}^{\mathrm{ns}}(g)=\frac1{4\pi i}\int_{(1-\sigma)}\widetilde{g}(1-z)\sum_{n^2\neq4\ell}\mathscr{L}_{n^2-4\ell}(1-z)\mathcal{J}_{n,\ell,h}(z)dz-\frac12\widetilde{g}(1)\mathcal{E}_{\ell,h}^{\mathrm{ns},0},
\]
where $\mathcal{E}_{\ell,h}^{\mathrm{ns},0}:=\operatorname*{Res}_{z=0}\sum_{n^2\neq4\ell}\mathscr{L}_{n^2-4\ell}(1-z)\mathcal{J}_{n,\ell,h}(z)$. 
Applying Lemma~\ref{lem:Jnlellh_fe} and $z\mapsto1-z$ yields
\[
\mathcal{G}_{\ell,h}^{\mathrm{ns}}(g)=\frac1{4\pi i}\int_{(\sigma)}\widetilde{g}(z)\sum_{n^2\neq4\ell}\mathscr{L}_{n^2-4\ell}(1-z)\mathcal{J}_{n,\ell,h_{1-z}}(z)dz-\frac12\widetilde{g}(1)\mathcal{E}_{\ell,h}^{\mathrm{ns},0},
\]
where $h_{1-z}(t)=h(t)\mathcal{G}_{\mu_t}(1-z)$.

For the singular part $n^2=4\ell$, the kernel $\mathcal{J}_{2m,m^2,h}(z)$ has a pole at $z=1/2$, so contour shifting gives an extra residue:
\begin{align*}
\mathcal{G}_{\ell,h}^{\mathrm{sing}}(g)&:=\frac1{4\pi i}\int_{(\sigma)}\widetilde{g}(1-z)\sum_{n^2=4\ell}\mathscr{L}_{n^2-4\ell}(1-z)\mathcal{J}_{n,\ell,h}(z)dz\\
&=\frac1{4\pi i}\int_{(1-\sigma)}\widetilde{g}(1-z)\sum_{n^2=4\ell}\mathscr{L}_{n^2-4\ell}(1-z)\mathcal{J}_{n,\ell,h}(z)dz\\
&\quad-\frac12\widetilde{g}(1)\mathcal{E}_{\ell,h}^{\mathrm{sing},0}-\frac12\widetilde{g}\left(\frac12\right)\operatorname*{Res}_{z=1/2}\sum_{n^2=4\ell}\mathscr{L}_{n^2-4\ell}(1-z)\mathcal{J}_{n,\ell,h}(z)\\
&=\frac1{4\pi i}\int_{(\sigma)}\widetilde{g}(z)\sum_{n^2=4\ell}\mathscr{L}_{n^2-4\ell}(z)\mathcal{J}_{n,\ell,h}(1-z)dz\\
&\quad-\frac12\widetilde{g}(1)\mathcal{E}_{\ell,h}^{\mathrm{sing},0}-\frac12\widetilde{g}\left(\frac12\right)\operatorname*{Res}_{z=1/2}\sum_{n^2=4\ell}\mathscr{L}_{n^2-4\ell}(1-z)\mathcal{J}_{n,\ell,h}(z),
\end{align*}
where $\mathcal{E}_{\ell,h}^{\mathrm{sing},0}:=\operatorname*{Res}_{z=0}\sum_{n^2=4\ell}\mathscr{L}_{n^2-4\ell}(1-z)\mathcal{J}_{n,\ell,h}(z)$.

We decompose $I_\ell(g;h)$ as
\[
I_\ell(g;h)=I_\ell^{\mathrm{ns}}(g;h)+I_\ell^{\mathrm{sing}}(g;h)+I_\ell^{0}(g;h)+I_\ell^{\mathrm{diag}}(g;h),
\]
with
\begin{align*}
I_\ell^{\mathrm{ns}}(g;h)&=\frac1{4\pi i}\int_{(\sigma)}\widetilde{g}(z)\sum_{n^2\neq4\ell}\mathscr{L}_{n^2-4\ell}(1-z)\mathcal{J}_{n,\ell,h_{1-z}}(z)dz,\\[4pt]
I_\ell^{\mathrm{sing}}(g;h)&=\frac1{4\pi i}\int_{(\sigma)}\widetilde{g}(z)\sum_{n^2=4\ell}\mathscr{L}_{n^2-4\ell}(z)\mathcal{J}_{n,\ell,h}(1-z)dz\\
&\quad-\frac12\widetilde{g}\left(\frac12\right)\operatorname*{Res}_{z=1/2}\sum_{n^2=4\ell}\mathscr{L}_{n^2-4\ell}(1-z)\mathcal{J}_{n,\ell,h}(z),\\[4pt]
I_\ell^{0}(g;h)&=-\frac12\widetilde{g}(1)\bigl(\mathcal{E}_{\ell,h}^{\mathrm{ns},0}+\mathcal{E}_{\ell,h}^{\mathrm{sing},0}\bigr),\\[4pt]
I_\ell^{\mathrm{diag}}(g;h)&=H_0(h)\Delta_\ell(g).
\end{align*}

By construction, $I_\ell^{\mathrm{ns}}(g;h)=I_\ell^{\vee,\mathrm{ns}}(g;h)$. It remains to prove
\[
I_\ell^{\mathrm{sing}}(g;h)+I_\ell^{0}(g;h)+I_\ell^{\mathrm{diag}}(g;h)-I_\ell^{\vee,\mathrm{sing}}(g;h)-I_\ell^{\vee,0}(g;h)-I_\ell^{\vee,\mathrm{diag}}(g;h)=\frac1{\pi}\int_\R h(t)\tau_t(\ell)
\mathcal{P}_g(t)dt,
\]
where 
\[
\mathcal{P}_g(t)=\widetilde g(1-2it)\frac{\zeta(1-4it)}{\zeta(1+2it)}
+\widetilde g(1)
+\widetilde g(1+2it)\frac{\zeta(1+4it)}{\zeta(1-2it)}.
\]

\vspace{0.5cm}
\emph{Step 5: Computation of $I_\ell^{0}(g;h)-I_\ell^{\vee,0}(g;h)$.}\quad
We first compute $I_\ell^0(g;h)$.
\begin{lemma}\label{lem:I-ell-0}
We have
\[
I_\ell^0(g;h)=\frac{\widetilde{g}(1)}{\pi}\int_{-\infty}^{\infty}h(t)\tau_t(\ell)dt.
\]
\end{lemma}

\begin{proof}
Recall that $I_\ell^0(g;h)=-\frac12\widetilde{g}(1)(\mathcal{E}_{\ell,h}^{\mathrm{ns},0}+\mathcal{E}_{\ell,h}^{\mathrm{sing},0})$, where
\[
\mathcal{E}_{\ell,h}^{\mathrm{ns},0}=\operatorname*{Res}_{z=0}\sum_{n^2\neq4\ell}\mathscr{L}_{n^2-4\ell}(1-z)\mathcal{J}_{n,\ell,h}(z),\qquad
\mathcal{E}_{\ell,h}^{\mathrm{sing},0}=\operatorname*{Res}_{z=0}\sum_{n^2=4\ell}\mathscr{L}_{n^2-4\ell}(1-z)\mathcal{J}_{n,\ell,h}(0).
\]

Write $n^2-4\ell=d^2$ and parametrize $n=\pm(a+b)$, $\ell=ab$ with $a\le b$. Using
\[
\operatorname*{Res}_{z=0}\mathscr{L}_{d^2}(1-z)=\begin{cases}
-1,&d\ge1\\[2pt]
-\frac12,&d=0,
\end{cases}
\]
and $\mathcal{J}_{-n,\ell,h}(0)=\mathcal{J}_{n,\ell,h}(0)$, we obtain
\[
\mathcal{E}_{\ell,h}^{\mathrm{ns},0}+\mathcal{E}_{\ell,h}^{\mathrm{sing},0}=-2\sum_{\substack{ab=\ell\\ a<b}}\mathcal{J}_{a+b,\ell,h}(0)-\mathbf1_{\ell=\square}\mathcal{J}_{2\sqrt\ell,\ell,h}(0).
\]

By Lemma \ref{lem:J-n-ell-t-0},
\[
\mathcal{J}_{a+b,ab,t}(0)=\frac1\pi\left[\left(\frac{a}{b}\right)^{it}+\left(\frac{b}{a}\right)^{it}\right],
\]
and integrating against $h$ gives $\mathcal{J}_{a+b,\ell,h}(0)=\frac1\pi\int_{-\infty}^{\infty}h(t)[(a/b)^{it}+(b/a)^{it}]dt$. Since 
\[
    \tau_t(\ell)
=\ell^{-it}\sigma_{2it}(\ell)
=\sum_{d\mid\ell}\left(\frac{d}{\ell/d}\right)^{it}
=\sum_{ab=\ell}\left(\frac{a}{b}\right)^{it},
\]
the sum over $a<b$ together with the square case yields the desired formula.
\end{proof}

Next we compute $I_\ell^{\vee,0}(g;h)$.

\begin{lemma}\label{lem:I-ell-vee-0}
\[
I_\ell^{\vee,0}(g;h)=-\frac1\pi\int_{-\infty}^{\infty}h(t)\tau_t(\ell)\left\{\widetilde{g}(1-2it)\frac{\zeta(1-4it)}{\zeta(1+2it)}+\widetilde{g}(1+2it)\frac{\zeta(1+4it)}{\zeta(1-2it)}\right\}dt.
\]
\end{lemma}

\begin{proof}
From the definition,
\[
I_\ell^{\vee,0}(g;h)=\frac1{4\pi i}\int_{(\sigma)}\widetilde{g}(z)\zeta(2-2z)\sum_{c\ge1}\sideset{}{^*}\sum_{x\bmod c}e\left(\frac{\ell\bar x}{c}\right)\lim_{y\to0^+}y^{z-1}\mathcal{J}_{h_{1-z}}\left(\frac{4\pi y\sqrt\ell}{c}\right)dz.
\]

Using the decomposition $\mathcal{J}_{h_{1-z}}(y)=\mathcal{Z}_{\ell,h}(z)J_{1-z}(y)+R_{\ell,h}(y)$ and the asymptotics $J_{1-z}(y)\sim y^{1-z}/2^{1-z}\Gamma(2-z)$ as $y\to0^+$, we obtain
\[
\lim_{y\to0^+}y^{z-1}\mathcal{J}_{h_{1-z}}\left(\frac{4\pi y\sqrt\ell}{c}\right)=\mathcal{Z}_{\ell,h}(z)\frac{(4\pi\sqrt\ell/c)^{1-z}}{2^{1-z}\Gamma(2-z)},
\]
with $\mathcal{Z}_{\ell,h}(z)$ given in \eqref{equ:Z-ell-h}. Substituting and evaluating the Ramanujan sum gives
\[
\sum_{c\ge1}\frac1c\sideset{}{^*}\sum_{x\bmod c}e\left(\frac{\ell\bar x}{c}\right)\left(\frac{4\pi\sqrt\ell}{c}\right)^{1-z}=\frac{(4\pi\sqrt\ell)^{1-z}}{\zeta(2-z)}\sigma_{z-1}(\ell).
\]

Simplifying using $\zeta(1-s)=G_0(s)\zeta(s)$ and $\Gamma(2-z)=(1-z)\Gamma(1-z)$ leads to
\[
I_\ell^{\vee,0}(g;h)=-\frac1{\pi i}\int_{(\sigma)}\widetilde{g}(z)\ell^{\frac{1-z}{2}}h\left(\frac{1-z}{2i}\right)\sigma_{z-1}(\ell)dz.
\]
Now set $z=1+2it$. Then $(1-z)/(2i)=-t$, $\ell^{(1-z)/2}\sigma_{z-1}(\ell)=\tau_t(\ell)$, $dz=2i\,dt$. Using evenness of $h$ and $\tau_{-t}(\ell)=\tau_t(\ell)$ gives the stated formula.
\end{proof}

Combining the two lemmas, we obtain
\[
I_\ell^0(g;h)-I_\ell^{\vee,0}(g;h)=\frac1\pi\int_{-\infty}^{\infty}h(t)\tau_t(\ell)\left\{\widetilde{g}(1-2it)\frac{\zeta(1-4it)}{\zeta(1+2it)}+\widetilde{g}(1)+\widetilde{g}(1+2it)\frac{\zeta(1+4it)}{\zeta(1-2it)}\right\}dt.
\]

It remains to prove
\begin{equation}\label{equ:maass-singular-diag}
I_\ell^{\mathrm{sing}}(g;h)+I_\ell^{\mathrm{diag}}(g;h)=I_\ell^{\vee,\mathrm{sing}}(g;h)+I_\ell^{\vee,\mathrm{diag}}(g;h).
\end{equation}

\vspace{0.5cm}
\emph{Step 6: Proof of the singular-diagonal equality.}\quad
If $\ell$ is not a square, there is nothing to prove. Assume $\ell=m^2$ with $m\ge1$. For $\sigma\in(-\frac12,0)$,
\begin{align*}
I_\ell^{\mathrm{sing}}(g;h)&=\frac{1}{2\pi i}\int_{(\sigma)}\widetilde{g}(z)\zeta(2z-1)\mathcal{J}_{2m,m^2,h}(1-z)dz-\widetilde{g}\left(\frac12\right)\operatorname*{Res}_{z=\frac12}\zeta(2z-1)\mathcal{J}_{2m,m^2,h}(z),\\[4pt]
I_\ell^{\mathrm{diag}}(g;h)&=\frac{1}{2\pi i}\int_{(\sigma)}H_0(h)\widetilde{g}(1-z)\zeta(2-2z)m^{z-1}dz,
\end{align*}
and
\begin{align*}
I_\ell^{\vee,\mathrm{sing}}(g;h)&=\frac{1}{2\pi i}\int_{(\sigma)}\widetilde{g}(z)\zeta(1-2z)\mathcal{J}_{2m,m^2,h_{1-z}}(z)dz,\\[4pt]
I_\ell^{\vee,\mathrm{diag}}(g;h)&=\frac{1}{2\pi i}\int_{(\sigma)}H_0(h_{1-z})\widetilde{g}(z)\zeta(2-2z)m^{z-1}dz.
\end{align*}

Denote $d(t):=t\tanh(\pi t)/\pi^2$. By Lemma \ref{lem:J-2m-m^2-t-fe},
\begin{equation}\label{equ:maass-critical-evaluation}
G_0(2z)\mathcal{J}_{2m,m^2,t}(z)=d(t)m^{-z}\mathcal{G}_{\mu_t}(z).
\end{equation}

\begin{lemma}\label{lem:maass-critical-singular}
We have the pointwise identity
\begin{multline}\label{equ:maass-critical-singular-pointwise}
\zeta(1-2z)\mathcal{J}_{2m,m^2,t}(z)+d(t)m^{z-1}\zeta(2-2z)\\
=\mathcal{G}_{\mu_t}(z)\Bigl(\zeta(2z-1)\mathcal{J}_{2m,m^2,t}(1-z)+d(t)m^{-z}\zeta(2z)\Bigr)
\end{multline}
and its averaged version
\begin{multline}\label{equ:maass-critical-singular-averaged}
\zeta(1-2z)\mathcal{J}_{2m,m^2,h}(z)+H_0(h)m^{z-1}\zeta(2-2z)\\
=\zeta(2z-1)\mathcal{J}_{2m,m^2,h_z}(1-z)+H_0(h_z)m^{-z}\zeta(2z),
\end{multline}
where $h_z(t)=h(t)\mathcal{G}_{\mu_t}(z)$.
\end{lemma}

\begin{proof}
The Riemann functional equation $\zeta(1-s)=G_0(s)\zeta(s)$ turns \eqref{equ:maass-critical-evaluation} into
\[
\zeta(1-2z)\mathcal{J}_{2m,m^2,t}(z)=d(t)m^{-z}\mathcal{G}_{\mu_t}(z)\zeta(2z).
\]
Replacing $z$ by $1-z$ in \eqref{equ:maass-critical-evaluation} and using $\mathcal{G}_{\mu_t}(z)\mathcal{G}_{\mu_t}(1-z)=1$ and $G_0(2z-1)G_0(2-2z)=1$ gives
\[
d(t)m^{z-1}\zeta(2-2z)=\mathcal{G}_{\mu_t}(z)\zeta(2z-1)\mathcal{J}_{2m,m^2,t}(1-z).
\]
Adding these two identities proves \eqref{equ:maass-critical-singular-pointwise}. Integrating against $h(t)$ yields \eqref{equ:maass-critical-singular-averaged}.
\end{proof}

From the pointwise identities, multiplying by $h(t)\mathcal{G}_{\mu_t}(1-z)$ and integrating gives
\[
\zeta(1-2z)\mathcal{J}_{2m,m^2,h_{1-z}}(z)=H_0(h)m^{-z}\zeta(2z),\qquad
H_0(h_{1-z})m^{z-1}\zeta(2-2z)=\zeta(2z-1)\mathcal{J}_{2m,m^2,h}(1-z).
\]
Since $\mathcal{G}_{\mu_t}(1/2)=1$, we have $H_0(h_{1/2})=H_0(h)$, and $\zeta(0)=-1/2$ implies
\[
\operatorname*{Res}_{z=1/2}\mathcal{J}_{2m,m^2,h}(z)=-H_0(h)m^{-1/2}.
\]

Substituting these into the expressions for $I_\ell^{\mathrm{sing}}$ and $I_\ell^{\mathrm{diag}}$, we obtain
\begin{align*}
I_\ell^{\mathrm{sing}}(g;h)&=\frac{1}{2\pi i}\int_{(\sigma)}\widetilde{g}(z)\zeta(2z-1)\mathcal{J}_{2m,m^2,h}(1-z)dz+\frac12\widetilde{g}\left(\frac12\right)H_0(h)m^{-1/2},\\
I_\ell^{\mathrm{diag}}(g;h)&=\frac{1}{2\pi i}\int_{(\sigma)}H_0(h)\widetilde{g}(z)\zeta(2z)m^{-z}dz-\frac12\widetilde{g}\left(\frac12\right)H_0(h)m^{-1/2}.
\end{align*}
Adding them cancels the residue terms:
\[
I_\ell^{\mathrm{sing}}(g;h)+I_\ell^{\mathrm{diag}}(g;h)=\frac{1}{2\pi i}\int_{(\sigma)}\widetilde{g}(z)\Bigl[\zeta(2z-1)\mathcal{J}_{2m,m^2,h}(1-z)+H_0(h)\zeta(2z)m^{-z}\Bigr]dz.
\]

Finally, applying \eqref{equ:maass-critical-singular-averaged} transforms the bracket into $\zeta(1-2z)\mathcal{J}_{2m,m^2,h_{1-z}}(z)+H_0(h_{1-z})m^{z-1}\zeta(2-2z)$, which is exactly the integrand of $I_\ell^{\vee,\mathrm{sing}}(g;h)+I_\ell^{\vee,\mathrm{diag}}(g;h)$. Hence
\[
I_\ell^{\mathrm{sing}}(g;h)+I_\ell^{\mathrm{diag}}(g;h)=I_\ell^{\vee,\mathrm{sing}}(g;h)+I_\ell^{\vee,\mathrm{diag}}(g;h).
\]

This completes the proof of Theorem \ref{thm:maass-voronoi}.

\section{From Voronoi to Functional Equation}
\subsection{Holomorphic case}

\begin{lemma}\label{lem:Fg-in-S}
Let $g\in\mathcal{S}$ and define
\begin{equation}\label{equ:F-g-x}
F_g(x)=\sumh_{f\in B_{2k}} a_f(\ell)\sum_{n=1}^\infty A_f(n)g(nx),\qquad x>0.
\end{equation}
Then $F_g\in\mathcal{S}$.
\end{lemma}

\begin{proof}
For $x>0$, set $g_x(t)=g(xt)$. By Theorem \ref{thm:holo-voronoi},
\[
F_g(x)=I_\ell(g_x)=I_\ell(\mathcal{T}_\mu g_x)=\sumh_{f\in B_{2k}} a_f(\ell)\sum_{n\ge1}A_f(n)\mathcal{T}_\mu g_x(n).
\]
A direct computation using Mellin inversion gives
\[
\mathcal{T}_\mu g_x(n)=\frac1x\,\mathcal{T}_\mu g\left(\frac{n}{x}\right).
\]
Thus
\[
F_g(x)=\frac1x\sumh_{f\in B_{2k}} a_f(\ell)\sum_{n\ge1}A_f(n)\mathcal{T}_\mu g\left(\frac{n}{x}\right).
\]
By Lemma \ref{lem:Tmug-props}, $\mathcal{T}_\mu g$ decays rapidly at infinity, so $F_g$ decays rapidly as $x\to0^+$. The rapid decay as $x\to\infty$ follows directly from $g\in\mathcal{S}$. Hence $F_g\in\mathcal{S}$.
\end{proof}

\begin{corollary}\label{cor:A-ell-meromorphic}
For each $\ell\ge1$, the averaged $L$-function
\[
A_\ell(s)=\sumh_{f\in B_{2k}} a_f(\ell)L(s,\operatorname{Sym}^2f)
\]
is meromorphic on $\mathbb{C}$.
\end{corollary}

\begin{proof}
For $\operatorname{Re}(s)>1$, applying Mellin transform to \eqref{equ:F-g-x} gives $\widetilde{F_g}(s)=A_\ell(s)\widetilde{g}(s)$. Since $F_g\in\mathcal{S}$, $\widetilde{F_g}$ is entire, and $A_\ell(s)=\widetilde{F_g}(s)/\widetilde{g}(s)$ extends meromorphically.
\end{proof}

\begin{corollary}\label{cor:A-ell-fe}
For any $g\in\mathcal{S}$,
\[
A_\ell(s)\widetilde{g}(s)=A_\ell(1-s)\mathcal{G}_\mu(1-s)\widetilde{g}(s).
\]
\end{corollary}

\begin{proof}[Proof of Theorem \ref{thm:holo-fe}]
Let $B_{2k}=\{f_1,\dots,f_d\}$. Then the vectors $(a_{f_1}(\ell))_{\ell\geqslant 1}, (a_{f_2}(\ell))_{\ell\geqslant 1},\dots, (a_{f_d}(\ell))_{\ell\geqslant 1}$ are linearly independent. So there exists integers $\ell_1<\cdots<\ell_d$ such that the matrix $(a_{f_i}(\ell_j))$ is invertible. Then
\[
\left(\frac{L(s,\operatorname{Sym}^2f_1)}{\|f_1\|},\dots,\frac{L(s,\operatorname{Sym}^2f_d)}{\|f_d\|}\right)=\frac{(4\pi)^{2k-1}}{\Gamma(2k-1)}(A_{\ell_1}(s),\dots,A_{\ell_d}(s))B^{-1}.
\]
So each $L(s,\operatorname{Sym}^2f_i)$ is a linear combination of $A_{\ell_j}(s)$. Since $A_\ell(s)=A_\ell(1-s)\mathcal{G}_\mu(1-s)$ by Corollary \ref{cor:A-ell-fe}, the functional equation follows.

Next we prove that $L(s,\operatorname{Sym}^2f)$ is entire. Let $s_0\in\mathbb{C}$ and pick $g\in\mathcal{S}$ with $\widetilde{g}(s_0)\neq0$ (If $\widetilde{g}(s_0)=0$, by Lemma \ref{lem:mellin-Schwartz}, the Mellin inversion of $\frac{\widetilde{g}(s)}{s-s_0}$ is also in $\mathcal{S}$). Since $\widetilde{F_g}(s)$ is entire, $A_\ell(s)=\widetilde{F_g}(s)/\widetilde{g}(s)$ is analytic at $s_0$. Hence each $A_\ell(s)$ is entire, and consequently each $L(s,\operatorname{Sym}^2f)$ is entire.

\end{proof}

\subsection{Maass case}

For $\operatorname{Re}(s)>1$, define
\[
L(s,\operatorname{Sym}^2 u)=\sum_{n=1}^{\infty}\frac{A_u(n)}{n^s}=\zeta(2s)\sum_{m=1}^{\infty}\frac{\lambda_u(m^2)}{m^s},
\]
and similarly 
\[
L(s,\operatorname{Sym}^2 t)=\sum_{n=1}^{\infty}\frac{A_t(n)}{n^s}=\zeta(2s)\sum_{m=1}^{\infty}\frac{\tau_t(m^2)}{m^s},
\]
By checking the Euler product, one sees that $L(s,\operatorname{Sym}^2,t)=\zeta(s+2it)\zeta(s)\zeta(s-2it)$.

\begin{lemma}[Continuous Eisenstein Voronoi formula]\label{lem:maass-eisenstein-voronoi}
For $g\in C_c^\infty(\mathbb{R}_{>0})$ and fixed $t\in\mathbb{R}$,
\[
\sum_{n=1}^{\infty}A_t(n)g(n)=\sum_{n=1}^{\infty}A_t(n)(\mathcal{T}_{\mu_t}g)(n)+\mathcal{P}_t(g),
\]
where 
\[
\mathcal{P}_g(t)=\widetilde g(1-2it)\frac{\zeta(1-4it)}{\zeta(1+2it)}
+\widetilde g(1)
+\widetilde g(1+2it)\frac{\zeta(1+4it)}{\zeta(1-2it)}.
\]
\end{lemma}

\begin{proof}
A standard Euler product computation gives $L(s,\operatorname{Sym}^2,t)=\zeta(s+2it)\zeta(s)\zeta(s-2it)$ for $\operatorname{Re}(s)>1$. Mellin inversion yields
\[
\sum_{n=1}^{\infty}A_t(n)g(n)=\frac{1}{2\pi i}\int_{(c)}\widetilde{g}(s)L(s,\operatorname{Sym}^2,t)ds.
\]
Since $A_t(n)=A_t(n)$, the same Dirichlet series appears for the right-hand side. Using $\widetilde{\mathcal{T}_{\mu_t}g}(s)=\widetilde{g}(1-s)\mathcal{G}_{\mu_t}(s)$ and $L(s,\operatorname{Sym}^2,t)=\mathcal{G}_{\mu_t}(1-s)L(1-s,\operatorname{Sym}^2,t)$, we obtain
\[
\sum_{n=1}^{\infty}A_t(n)(\mathcal{T}_{\mu_t}g)(n)=\frac{1}{2\pi i}\int_{(1-c)}\widetilde{g}(w)L(w,\operatorname{Sym}^2,t)dw.
\]
Shifting the contour from $\operatorname{Re}(s)=c$ to $\operatorname{Re}(s)=1-c$ picks up residues at $s=1-2it,1,1+2it$, giving $\mathcal{P}_t(g)$. Hence the identity holds.
\end{proof}

\begin{corollary}\label{cor:maass-voronoi-discrete}
Let $h$ be Kuznetsov-admissible with $h(\pm i/2)=0$. Then for $g\in\mathcal{S}$,
\[
\sum_{u\in\mathcal{B}}\omega_u h(t_u)\lambda_u(\ell)\sum_{n=1}^{\infty}A_u(n)g(n)=\sum_{u\in\mathcal{B}}\omega_u h(t_u)\lambda_u(\ell)\sum_{n=1}^{\infty}A_u(n)\mathcal{T}_{\mu_u}g(n).
\]
\end{corollary}

\begin{lemma}[Spectral projector]
\label{lem:maass-spectral-projector}
Fix a cuspidal spectral parameter $t_0$, and put
\[
\mathcal B[t_0]:=\{u\in\mathcal B:t_u^2=t_0^2\}.
\]
There is an even entire function $H_{t_0}$, rapidly decreasing on every fixed
horizontal strip, such that
\[
H_{t_0}(t_u)=
\begin{cases}
1,&u\in\mathcal B[t_0],\\
0,&u\notin\mathcal B[t_0],
\end{cases}
\qquad
H_{t_0}(\pm i/2)=0.
\]
\end{lemma}

\begin{proof}
The packet $\mathcal B[t_0]$ is finite by finite multiplicity of the
cuspidal Laplace spectrum.  Weyl's law implies that the set
$\{\pm t_u:u\notin\mathcal B[t_0]\}$ has exponent of convergence at most
$2$.  Hence the Weierstrass product theorem gives an even entire function
$Q_{t_0}$ of order $<4$ such that
\[
Q_{t_0}(\pm t_u)=0\quad(u\notin\mathcal B[t_0]),
\qquad
Q_{t_0}(\pm t_0)\neq0.
\]
No cuspidal spectral parameter is equal to $\pm i/2$.  Therefore
\[
H_{t_0}(t):=
\frac{(t^2+1/4)^2Q_{t_0}(t)e^{-t^4}}
{(t_0^2+1/4)^2Q_{t_0}(t_0)e^{-t_0^4}}
\]
has the required values and double zeros at $\pm i/2$.  Since
$Q_{t_0}$ has order $<4$, the factor $e^{-t^4}$ gives rapid decay on
every fixed horizontal strip.
\end{proof}

\begin{corollary}\label{cor:single-average-sum}
Let $t_0$ be a cuspidal spectral parameter. Then for any $g\in\mathcal{S}$, we have
\[
   \sum_{u\in\mathcal{B}[t_0]} \omega_u \lambda_u(\ell)\sum_{n=1}^{\infty}A_u(n)g(n)=\sum_{u\in\mathcal{B}[t_0]}\omega_u \lambda_u(\ell) \sum_{n=1}^{\infty}A_u(n)\mathcal{T}_{\mu_u}g(n).
\]
\end{corollary}

\begin{proof}
It follows by taking $h(t)=H_{t_0}(t)$.
\end{proof}

Let $t_0$ be a cuspidal spectral parameter. For $g\in\mathcal{S}$, define
\begin{equation*}
    F_{\ell,g,t_0}(x)=\sum_{u\in\mathcal{B}[t_0]} \omega_u \lambda_u(\ell)\sum_{n=1}^{\infty}A_u(n)g(nx).
\end{equation*}
for $x>0$. Then
\begin{lemma}\label{lem:maass-Fg-schwartz}
For $g\in\mathcal{S}$ and $t_0$ as above, we have
\[
F_{\ell,g,t_0}(x)=\frac{1}{x}\sum_{u\in\mathcal{B}}\omega_u \lambda_u(\ell)\sum_{n=1}^{\infty}A_u(n)\omega_u \lambda_u(\ell)\mathcal{T}_{\mu_{t_0}}g\left(\frac{n}{x}\right)
\]
and $F_{\ell,g,t_0}\in\mathcal{S}$.
\end{lemma}

\begin{proof}
Set $g_x(y)=g(xy)$. Notice that for $u\in\mathcal{B}[t_0]$, $\mathcal{T}_{\mu_u}=\mathcal{T}_{\mu_{t_0}}$. By Corollary \ref{cor:single-average-sum} and $\mathcal{T}_{\mu_u}g_x(n)=\frac{1}{x}\mathcal{T}_{\mu_u}g(n/x)$, we have
\[
F_{\ell,g,t_0}(x)=\frac{1}{x}\sum_{u\in\mathcal{B}}\omega_u \lambda_u(\ell)\sum_{n=1}^{\infty}A_u(n)\mathcal{T}_{\mu_u}g\left(\frac{n}{x}\right)\frac{1}{x}\sum_{u\in\mathcal{B}}\omega_u \lambda_u(\ell)\sum_{n=1}^{\infty}A_u(n)\mathcal{T}_{\mu_{t_0}}g\left(\frac{n}{x}\right).
\]
Rapid decay as $x\to\infty$ follows from $g\in\mathcal{S}$, and as $x\to0^+$ from the rapid decay of $h$. Hence $F_{g,\ell,h}\in\mathcal{S}$.
\end{proof}

Define
\[
    A_{\ell,g,t_0}(s)=\sum_{u\in\mathcal B[t_0]}\omega_u \lambda_u(\ell)L(s,\operatorname{Sym}^2u),\quad\mathrm{Re}(s)>1.
\]

\begin{corollary}\label{cor:maass-avg-meromorphic}
For every $g\in\mathcal S$, the function $A_{\ell,g,t_0}(s)$ is an entire function satifying the functional equation
\[
    A_{\ell,g,t_0}(s)=\mathcal{G}_{\mu_{t_0}}(1-s)A_{\ell,g,t_0}(1-s).
\]
\end{corollary}

\begin{proof}
Taking Mellin transforms for the two different expressions of $F_{\ell,g,t_0}$ yields
\[
    \widetilde{F_{\ell,g,t_0}}(s)=\widetilde{g}(s)A_{\ell,g,t_0}(s)=\widetilde{\mathcal{T}_{\mu_{t_0}}g}(s)A_{\ell,g,t_0}(1-s)
\]
So the desired functional equation holds. By
\[
    A_{\ell,g,t_0}(s)=\frac{\widetilde{{F_{\ell,g,t_0}}}(s)}{\widetilde{g}(s)},
\]
we see that it is meromorphic in $\mathbb{C}$. Noting that for every complex number $s_0$, we can choose $g\in\mathcal{S}$ such that $\widetilde{g}(s_0)\neq 0$. Hence $A_{\ell,g,t_0}(s)$ is entire.
\end{proof}

\begin{proof}[Proof of Theorem \ref{thm:maass-fe}]
Let $u\in\mathcal{B}[t_0]$. Notice that the set $\mathcal{B}[t_0]$ is a finite set. By a similar argument as in the proof of Theorem \ref{thm:holo-fe}, we see that $L(s,\operatorname{Sym}^2u)$ is a linear combination of $A_{\ell_1,g,t_0}(s), A_{\ell_2,g,t_0}(s),\dots,A_{\ell_r,g,t_0}(s)$. So $L(s,\operatorname{Sym}^2u)$ is entire and satisfies the desired functional equation.
\end{proof}

\bibliographystyle{plain}
\bibliography{biblio}

\end{document}